\documentclass[11pt, a4paper,reqno]{amsart}

\usepackage[table]{xcolor}
\usepackage{amsthm}
\usepackage{microtype}
\usepackage{amssymb}
\usepackage{tikz-cd}
\usepackage{enumitem}
\usepackage{mathtools}
\usepackage{stackrel}
\usepackage{MnSymbol}
\usepackage{array}
\usepackage{booktabs}
\usepackage{multirow}
\usepackage{tabu}
\usepackage{xcolor}
\usepackage{stmaryrd}

\usepackage{etoolbox}
\apptocmd{\thebibliography}{\setlength{\itemsep}{5pt}}{}{}

\usepackage{pifont}

\newcommand{\xmark}{\ding{55}}

\usepackage{caption}
\usepackage{subcaption}

\renewcommand{\emptyset}{\font\cmsy = cmsy10 at 10pt
 \hbox{\cmsy \char 59}
}

\usepackage{tikzit}

\tikzstyle{blackDot}=[inner sep=0mm, minimum size=1.5mm, draw=black, shape=circle, draw=white, fill=black, line width=0.2mm]
\tikzstyle{boxSquare}=[fill=white, draw=black, shape=rectangle, minimum width=7mm, minimum height=5mm, font={\scriptsize}]
\tikzstyle{boxBroad}=[fill=white, draw=black, shape=rectangle, minimum width=10mm, minimum height=5mm, font={\scriptsize}]
\tikzstyle{boxSmall}=[fill=white, draw=black, shape=rectangle, minimum width=5mm, minimum height=5mm, font={\scriptsize}]
\tikzstyle{emptyText}=[fill=none, draw=none, shape=circle, font={\tiny}]
\tikzstyle{emptyTextScripsize}=[fill=none, draw=none, shape=circle, font={\scriptsize}]
\tikzstyle{state}=[fill=white, draw=black, regular polygon, regular polygon sides=3, minimum width=0.8cm, shape border rotate=180, inner sep=0pt, font={\scriptsize}]
\tikzstyle{stateLarge}=[fill=white, draw=black, regular polygon, regular polygon sides=3, minimum width=1cm, shape border rotate=180, inner sep=0pt, font={\scriptsize}]
\tikzstyle{box very broad}=[fill=white, draw=black, shape=rectangle, minimum width=20mm]
\tikzstyle{stateReverse}=[fill=white, draw=black, regular polygon, regular polygon sides=3, minimum width=0.8cm, inner sep=0pt, font={\scriptsize}]
\tikzstyle{whiteDot}=[fill=white, draw=black, shape=circle, minimum size=1.2mm, shape=circle, inner sep=0mm]

\tikzstyle{edge}=[fill=none, draw=black, line width=0.65pt, -]
\tikzstyle{white edge}=[-, draw=white, line width=1.5mm]
\tikzstyle{ddashed}=[-, draw=black, dashed]
\tikzstyle{doubleline}=[-, double, draw=black, fill=none]
\tikzstyle{redLine}=[-, draw=red]
\tikzstyle{blueLine}=[-, draw=blue]

\usepackage[utf8]{inputenc}
\usepackage[T1]{fontenc}

\usepackage{tgpagella}
\usepackage[T1]{fontenc}

\usepackage[colorlinks, linkcolor=black,citecolor=black,linktocpage,hypertexnames=true,pdfpagelabels]{hyperref}
\usepackage{cleveref}

\newtheorem{theorem}{Theorem}
\newtheorem{lemma}[theorem]{Lemma}
\newtheorem{corollary}[theorem]{Corollary}
\newtheorem{proposition}[theorem]{Proposition}

\newtheorem{definition}[theorem]{Definition}
\newtheorem{question}[theorem]{Question}

\newtheorem{assumption}[theorem]{Assumption}
\newtheorem{notation}[theorem]{Notation}

\theoremstyle{definition}

\newtheorem{example}[theorem]{Example}
\newtheorem{remark}[theorem]{Remark}

\newcommand\diff[1]{\, \mathrm{d}#1}

\newcommand{\Del}[1]{\mathrm{\mathsf{del}}_{#1}}
\newcommand{\Copy}[1]{\mathrm{\mathsf{copy}}_{#1}}

\usepackage[margin=3cm]{geometry}

\newcommand{\cat}[1]{\mathbf{#1}}

\newcommand{\In}{\mathsf{in}}
\newcommand{\Out}{\mathsf{out}}
\newcommand{\norm}{\mathsf{norm}}

\newcommand{\An}{\mathsf{An}}
\newcommand{\Dec}{\mathsf{Dec}}
\newcommand{\Cut}{\mathsf{Cut}}

\makeatletter
\newcommand{\thickhline}{
    \noalign {\ifnum 0=`}\fi \hrule height 1.5pt
    \futurelet \reserved@a \@xhline
}
\newcolumntype{"}{@{\hskip\tabcolsep\vrule width 1pt\hskip\tabcolsep}}
\makeatother

\begin{document}

\title{The \texorpdfstring{$d$}{d}-separation criterion in Categorical Probability}

\author{Tobias Fritz}
\address{Department of Mathematics, Technikerstr.\ 13,  A-6020 Innsbruck, Austria}
\email{tobias.fritz@uibk.ac.at}

\author{Andreas Klingler}
\address{Institute for Theoretical Physics, Technikerstr.\ 21a,  A-6020 Innsbruck, Austria}
\email{andreas.klingler@uibk.ac.at}

\date{\today}

\subjclass[2020]{18M30, 62A09 (Primary), 18M35, 60A05, 62D20 (Secondary)}

\thanks{\emph{Acknowledgments.} We thank Rob Spekkens for the original idea behind the definition of categorical $d$-separation, Patrick Forr\'e for discussion and pointers to the literature, and Areeb Shah-Mohammed for pointing out a typo in \Cref{def:stringDiagram}. AK acknowledges support from the Stand Alone project P33122-N of the Austrian Science Fund (FWF) and funding of the Austrian Academy of Sciences
(\"OAW) through a DOC scholarship.}

\begin{abstract}%
The $d$-separation criterion detects the compatibility of a joint probability distribution with a directed acyclic graph through certain conditional independences. In this work, we study this problem in the context of categorical probability theory by introducing a categorical definition of causal models, a categorical notion of $d$-separation, and proving an abstract version of the $d$-separation criterion. This approach has two main benefits. First, categorical $d$-separation is a very intuitive criterion based on topological connectedness. Second, our results apply both to measure-theoretic probability (with standard Borel spaces) and beyond probability theory, including to deterministic and possibilistic networks. It therefore provides a clean proof of the equivalence of local and global Markov properties with causal compatibility for continuous and mixed random variables as well as deterministic and possibilistic variables.
\end{abstract}

\maketitle

\setcounter{tocdepth}{1}
\tableofcontents

\section{Introduction}

The $d$-separation criterion \cite{Pe09} is a necessary and sufficient condition for the compatibility of a probability distribution with a causal
structure in the form of a directed acyclic graph (DAG). 
It states that a joint probability distribution of a collection of random variables is compatible with the DAG---in the sense that each of its nodes is one of the given variables, and each arrow denotes the possibility of causal influence---if and only if the distribution satisfies a list of conditional independence relations encoded in the structure of the DAG.

In this paper, we study this causal compatibility problem in the framework of \emph{categorical probability theory}. We elaborate on the framework of generalized causal models recently proposed in \cite{Fr22}, introduce a categorical notion of $d$-separation, and prove a categorical generalization of the $d$-separation criterion.

The framework of generalized causal models involves freely generated categories called \emph{free Markov categories}~\cite{Fr22}. 
Starting from a set of morphisms as building blocks (representing the basic causal mechanisms), we construct a morphism in this category by assembling these blocks consistently. 
More precisely, a morphism in such a free category is a \emph{string diagram} consisting of wires and boxes. In our context, these diagrams represent the causal models. Each wire corresponds to a local random variable and each box to a causal mechanism generating one or several new output variables from its input ones.
A morphism in any Markov category is \emph{compatible} with such a causal model if it can be decomposed in the form specified by the causal model; this generalizes the standard factorization definition of Bayesian networks~\cite{Fo13}.

Further, we define a categorical notion of $d$-separation in terms of the string diagrams. A causal structure displays $d$-separation of a collection of outputs $\mathcal{Z}$ if it becomes disconnected upon removing the corresponding wires in $\mathcal{Z}$. We show that this is equivalent to classical $d$-separation for the class of causal models for which the latter is defined, namely those corresponding to DAGs. 

Finally, we prove an abstract version of the $d$-separation criterion for our categorical notions of causal model and $d$-separation. We show that a given distribution, possibly depending on additional input variables, is compatible with a causal structure if and only if it displays conditional independence between sets of variables whenever the causal model displays $d$-separation for the corresponding sets of wires. This result goes beyond the classical result for ``distributions''; it holds for arbitrary morphisms in any suitable Markov category $\cat{C}$. A central structural ingredient in the proof is the assumption that $\cat{C}$ has conditionals \cite{Fr20}. Intuitively, this property states that every Markov kernel with multiple output variables arises by generating one output variable after the other, namely as a function of the input variables and previously generated outputs.
For example, Markov kernels between finite sets have conditionals since every conditional probability distribution $p(x,y \, | \, a)$ factorizes into
$$ p(x,y \, | \, a) = p(x \, | \, y,a) \cdot p(y \, | \, a).$$
The conditional distribution $p(x \,| \, y,a)$ arises by solving this equation for every $x,y$, and $a$ and taking it to be arbitrary when $p(y \, | \, a) = 0$.
Conditionals also exist for continuous or mixed random variables (technically for Markov kernels between standard Borel spaces, but not arbitrary measurable spaces) as well as in the Markov category of Gaussian kernels. 
Moreover, we will show that Markov categories with conditionals also exist in settings that do not involve numerical probabilities. This implies that our results apply equally easily to all of these cases.

This approach leads to new insights into the structure of $d$-separation and generalizes the classical result, to the best of our knowledge, in four notable directions:
\begin{itemize}[label={--}]
	\item It generalizes the $d$-separation criterion in the pure DAG setting to a larger class of causal models.  First, it gives a criterion for the compatibility of \emph{Markov kernels} with a causal structure rather than merely for probability distributions. This is because Markov categories formalizing probabilities have Markov kernels as their morphisms, which therefore become the basic primitives of our formalism. Second, the local mechanisms in the causal structures are allowed to have multiple outputs and to appear multiple times.
	While all those generalizations can also be included in the DAG setting via suitable workarounds, as considered for example by \cite{Fo21}, string diagrams contain them already intrinsically.
	\item Categorical $d$-separation on string diagrams is more general and conceptually simpler than ``classical'' $d$-separation on DAGs since it only relies on three intuitive operations on string diagrams: \emph{marginalization}, \emph{removing wires}, and \emph{checking disconnectedness}.
	\item It provides a uniform proof of the equivalence between $d$-separation and causal compatibility for discrete variables, continuous variables (or rather arbitrary variables taking values in standard Borel spaces), and Gaussian variables. This follows from the fact that we reason abstractly and even more generally in any Markov category with conditionals. Together with the first point, this approach, therefore, generalizes the $d$-separation criterion on the pure DAG setting for continuous variables shown for distribution with a density \cite{La90,La96}, which was later extended to arbitrary variables with values in standard Borel spaces in \cite{Fo17}.
	\item The abstract reasoning using Markov categories implies that the $d$-separation criterion also applies to non-probabilistic settings, for example to networks modeling deterministic or possibilistic variables and causal mechanisms.
\end{itemize}

This paper is organized as follows. In \Cref{sec:overview}, we give a more detailed non-technical overview of the main concepts of this paper, including causal models as string diagrams and categorical $d$-separation.
In \Cref{sec:MarkovCat}, we recall and explain the definitions of Markov categories and gs-monoidal categories. In \Cref{sec:FreeMarkov}, we review the construction of free gs-monoidal and free Markov categories, leading to generalized causal models and causal compatibility for morphisms in Markov categories. In \Cref{sec:Conditional}, we present the notions of conditionals and conditional independence. In \Cref{sec:decidingComp}, we introduce a categorical notion of $d$-separation and prove the main results of this paper, namely the equivalence with classical $d$-separation (\Cref{pro:dSepCoincide}) and our abstract version of the $d$-separation criterion for causal compatibility (\Cref{thm:causalCompdSep}).

\section{Causal models and Markov categories}
\label{sec:overview}

In this paper, we study causal models, in the sense of Bayesian networks, from a categorical perspective. In order to make this accessible to readers without a formal background in category theory, we outline this paper's main concepts and results in this section.

\Cref{ssec:BayesianNetworksString} gives a non-technical introduction to the string diagrammatic formalism representing causal models and its relation to the classical DAG formalism. We then present the concept of causal compatibility in \Cref{ssec:CausalCompatibility} as a functorial property.
In \Cref{ssec:dsep}, we present categorical $d$-separation as a statement about the connectedness of the string diagram and explain our result on the $d$-separation criterion in \Cref{ssec:dsepCrit}. Therefore, the string diagrammatic approach opens the door for a new perspective on $d$-separation and its connection to causal compatibility.

\subsection{Extending Bayesian networks with string diagrams}
\label{ssec:BayesianNetworksString}

Traditionally, the definition of Bayesian networks relies on the concept of directed acyclic graphs (DAG) since such a graph encodes the underlying causal structure. Each node $v \in V(G)$ of a DAG $G$ is associated with a random variable $X_v$, and each directed arrow $w \to v$ is associated with a direct possible causal dependence of the variable $X_v$ on $X_w$.
Formally, if we index the nodes by $1,\ldots,n$, then this means that a joint probability distribution $P$ is compatible with $G$ if $P$ factorizes into
\begin{equation*}
P\big(X_1, \ldots, X_n\big) = \prod_{i=1}^{n} P\big(X_i \, | \, \textsf{Pa}(X_i)\big)
\end{equation*}
where $\textsf{Pa}(X_i) = \{X_j: G \text{ contains the arrow } j \to i\}$ is the set of parents of $X_i$ relative to $G$.

In the categorical framework, we can refine the notion of Bayesian networks via string diagrams, an idea that extends the categorical approach developed by Fong \cite{Fo13} and others \cite{Ri20,Ri21,Ga22, Ja19}.
These diagrams arise naturally in \emph{categorical probability}, which refers more generally to recent work on axiomatizing probability theory with simple, algebraic rules which avoid the low-level machinery of measure theory by hiding it in the proofs of the relevant axioms \cite{Fr20}.

In categorical probability, the basic primitives are conditional probability distributions, also known as Markov kernels. A Markov kernel is represented as a morphism $f: A \to X$, depicted as
\begin{equation*}
\input{figures/singleOutputMorphism.tikz}
\end{equation*}
Here, $A$ represents the input space and $X$ the output space. The morphism $f$ can be understood as a map that assigns a probability distribution on $X$ to every point $a \in A$ in a measurable way. Depending on the particular context, $f$ can be a stochastic matrix representing a Markov kernel between finite sets, a linear map with added Gaussian noise, or an arbitrary measurable Markov kernel.
For example, a finite stochastic matrix is defined as a map
\begin{align*}
	f \: : \: & X \times A \to [0,1] \\
		& (x,a) \mapsto f(x \, | \, a)
\end{align*}
such that $f(- \, | \, a)$ is a probability distribution on $X$ for every $a \in A$. When choosing $A$ to be a singleton set $I$, a morphism $p: I \to X$ represents a probability distribution, depicted as
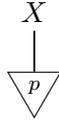
\begin{equation*}
\begin{tikzpicture}
	\begin{pgfonlayer}{nodelayer}
		\node [style=none] (0) at (0, 1) {};
		\node [style=state] (1) at (0, -0.5) {$p$};
		\node [style=none] (2) at (0, 1.5) {$X$};
	\end{pgfonlayer}
	\begin{pgfonlayer}{edgelayer}
		\draw [style=edge] (0.center) to (1);
	\end{pgfonlayer}
\end{tikzpicture}

\end{equation*} 
Here, $I$ is not drawn in the diagram, the upper wire represents the probability space $X$, and the box $p$ the distribution, considered now as a function with no input and one (random) output in $X$. Therefore, probability distributions are considered as special cases of the basic primitives in categorical probability by setting the input space $A$ to be a trivial object $I$. 

These morphisms can be composed into new ones. For example, composing $f:X \to Y$ with $p: I \to X$, pictorially
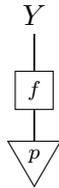
\begin{equation*}
\input{figures/composition.tikz} 
\end{equation*}
gives rise to a new distribution of a random variable with values in $Y$. In the concrete setting, the composition is given by the Chapman--Kolmogorov equation, which reads for discrete distributions as
\begin{equation*} 
(f \circ p)(y) \coloneqq \sum_{x \in X} f(y\, | \, x) \cdot p(x)
\end{equation*}
and for Markov kernels on arbitrary measurable spaces as
\begin{equation*}
(f \circ p)(A) \coloneqq \int_{X} f(A \, | \, x) \, p(\mathrm{d} x). 
\end{equation*}
where $A \in \Sigma_Y$, with $\Sigma_Y$ the $\sigma$-algebra of the measurable space $Y$.

Each random variable can be copied or marginalized over, for which there exists special string diagrammatic notation:
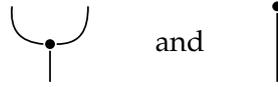
\begin{equation*}
\input{figures/copy.tikz} \qquad \text{ and } \qquad \begin{tikzpicture}
	\begin{pgfonlayer}{nodelayer}
		\node [style=blackDot] (0) at (0, 1) {};
		\node [style=none] (3) at (0, -1) {};
	\end{pgfonlayer}
	\begin{pgfonlayer}{edgelayer}
		\draw [style=edge] (3.center) to (0);
	\end{pgfonlayer}
\end{tikzpicture}

\end{equation*} 

These maps satisfy manipulation rules formalizing basic properties of information flow, for example:
\begin{enumerate}[label=(\roman*)]
	\item When marginalizing the output of a box, the action of the box becomes irrelevant and therefore the box can be removed from the diagram, pictorially
\begin{equation}
\label{eq:delBox}
 \begin{tikzpicture}
	\begin{pgfonlayer}{nodelayer}
		\node [style=boxSquare] (2) at (0, 0) {$f$};
		\node [style=blackDot] (4) at (0, 1.5) {};
		\node [style=none] (5) at (0, -1.5) {};
	\end{pgfonlayer}
	\begin{pgfonlayer}{edgelayer}
		\draw [style=edge] (5.center) to (4);
	\end{pgfonlayer}
\end{tikzpicture}
 \ = \ \begin{tikzpicture}
	\begin{pgfonlayer}{nodelayer}
		\node [style=blackDot] (4) at (0, 1.5) {};
		\node [style=none] (5) at (0, -1.5) {};
	\end{pgfonlayer}
	\begin{pgfonlayer}{edgelayer}
		\draw [style=edge] (5.center) to (4);
	\end{pgfonlayer}
\end{tikzpicture}

\end{equation}
	This also applies in the case of multiple outputs since the formalism allows viewing multiple wires as a single one (which corresponds to the operation of taking products of measurable spaces).
	\item When marginalizing an output of the copy map, the resulting map acts like an identity, pictorially
\begin{equation}
\label{eq:delCopy}
	\input{figures/leftCopyContraction.tikz} \ = \ \input{figures/rightCopyContraction.tikz} \ = \ \begin{tikzpicture}
	\begin{pgfonlayer}{nodelayer}
		\node [style=none] (0) at (0, -1) {};
		\node [style=none] (1) at (0, 2) {};
	\end{pgfonlayer}
	\begin{pgfonlayer}{edgelayer}
		\draw [style=edge] (0.center) to (1.center);
	\end{pgfonlayer}
\end{tikzpicture}

\end{equation}
\end{enumerate}

We formally axiomatize the semantics of the resulting string diagram calculus in \Cref{def:MarkovCat} as the definition of a \emph{Markov category}.

Each flavour of probability has its own Markov category. There is a Markov category for discrete probabilities (called $\cat{FinStoch}$), one for Gaussian probability ($\cat{Gauss}$), one for probability theory on standard Borel spaces ($\cat{BorelStoch}$), and one for probability theory on arbitrary measurable spaces ($\cat{Stoch}$).

But there are also non-probabilistic Markov categories. These include Markov categories in which every morphism is deterministic (for example $\cat{FinSet}$) and possibilistic Markov categories (for example $\cat{FinSetMulti}$), in which the non-determinism comes merely in the form a distinction between which outcomes are possible and which ones are not: for a morphism $f: X \to Y$, we have $f(y \, | \, x) = 1$ or $f(y \, | \, x) = 0$ according to whether a certain outcome $y \in Y$ is possible on an input $x \in X$ or not.

The causal structures themselves are also morphisms in yet another class of Markov categories, the so-called \emph{free Markov categories}. These form our framework for causal structures which generalize DAGs.
Rather than Markov kernels, the morphisms in a free Markov category are the string diagrams themselves, i.e.\ all ``networks'' that can be built by wiring together a set of boxes, similar to how an electrical circuit is obtained by wiring together electrical components.
In this way, string diagrams constitute generalized causal models.
In particular, we will see that string diagrams can represent arbitrary DAG causal models.
Consider for example this DAG:
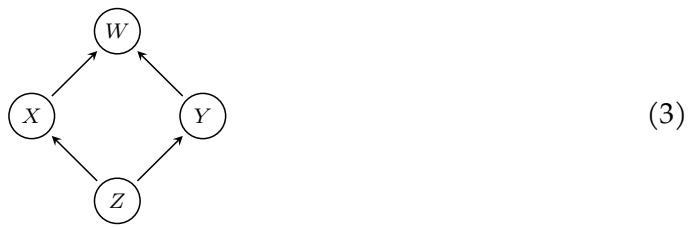
\begin{equation}
\label{eq:DAGScenarioDAG}
\begin{tikzpicture}[scale=0.75, shift={(-5,-3)}]
	\draw[fill=white, line width=0.2mm] (0,0) circle (0.8cm) node{\scriptsize $Z$};
	\draw[fill=white, line width=0.2mm] (-3,3) circle (0.8cm) node{\scriptsize $X$};
	\draw[fill=white, line width=0.2mm] (3,3) circle (0.8cm) node{\scriptsize $Y$};
	\draw[fill=white, line width=0.2mm] (0,6) circle (0.8cm) node{\scriptsize $W$};
	
	\draw[-stealth,line width=0.2mm,rotate=45] (0,1) -- (0,3.25);
	\draw[-stealth,line width=0.2mm,rotate=-45] (0,1) -- (0,3.25);
	
	\draw[-stealth,line width=0.2mm,rotate around={45:(3,3)}] (3,4) -- (3,6.25);
	\draw[-stealth,line width=0.2mm,rotate around={-45:(-3,3)}] (-3,4) -- (-3,6.25);	
	
\end{tikzpicture}
\end{equation}
As a string diagram, this causal structure looks like this:
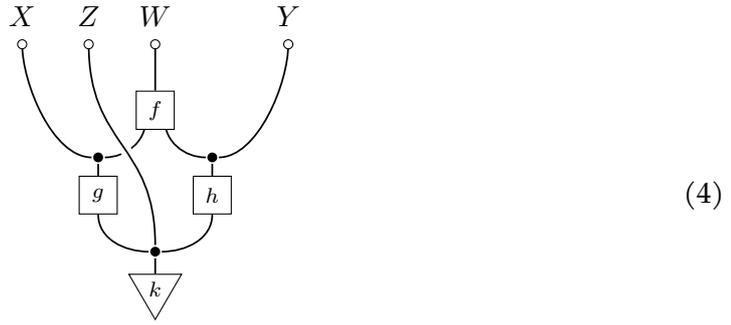
\begin{equation}
\label{eq:DAGScenario}
\input{figures/DAGscenario1WithoutType.tikz}
\end{equation}
By definition, a probability distribution $P(X,Y,Z,W)$ is compatible with this structure if
$$ P(X,Y,Z,W) = P(W \, | \, X,Y) \cdot P(X \, | \, Z) \cdot P(Y \, | \, Z) \cdot P(Z).$$
Each box in the string diagram represents a placeholder for a conditional distribution. For example, the box $f$ stands for the conditional distribution $P(W \, | \, XY)$.
Each loose wire\footnote{Note that we use the term ``wire'' as referring to an entire connected piece of circuitry, i.e.~traversing a black dot in the diagram does not leave the wire.} represents a random variable where the wire's name indicates the variable's name. Further, each variable has a corresponding type, a placeholder for the measurable space in which the variable takes values. Unless necessary, we will not explicitly mention the type of each variable.

In our setting, some wires are connected to an ``output'' representing a variable that is ``observed'' rather than marginalized over, as indicated by the white dots. Note that every variable then becomes an output in at most one way. While the above string diagram connects every wire to an output, the following string diagram also contains ``unobserved'' variables:
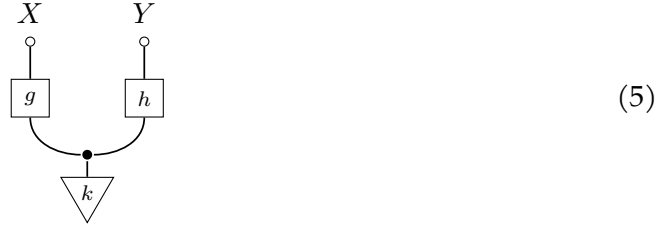
\begin{equation}
\label{eq:DAGScenarioXYZHidden}
\input{figures/DAGScenarioXYZHidden.tikz}
\end{equation}
In the setting of discrete probability, a distribution $P(X,Y)$ is therefore compatible with this structure, if 
$$P(X, Y) = \sum_Z P(X\, | \, Z) \cdot P(Y\, | \, Z) \cdot P(Z), $$
where the marginalized variable $Z$ is represented by the wire not connected to the output.

In certain situations, we only consider causal models where every wire is connected to an output. Throughout the paper, we call such diagrams \emph{pure blooms} (see \Cref{def:purebloom}). Unless stated otherwise, we will denote the output by the name of the wire connected to the output. We will define certain operations, like $d$-separations, only with respect to output wires, highlighting that we cannot address latent variables (i.e.\ wires that are not connected to an output) in general.

Causal models given by a DAG correspond to pure bloom string diagrams where each box has precisely one output wire, and the input wires represent the parents of the node in the DAG. The following table explains the relation between nodes in a DAG, boxes in the string diagram and the corresponding conditional probability distribution.

\medskip
\begin{center}
\renewcommand{\arraystretch}{1.1}
\setlength{\tabcolsep}{0pt}
\begin{tabu}{p{0.2cm} p{0.5cm} |[1pt] >{\centering\arraybackslash}p{3.8cm} | >{\centering\arraybackslash}p{3.8cm} | >{\centering\arraybackslash}p{3.8cm} p{0.2cm} }
& \multirow{5}{*}{}\\[-0.5cm]
& \multicolumn{1}{p{0.3cm} |[1pt]}{} & \scriptsize DAG & \scriptsize string diagram & \scriptsize conditional distribution &\\
\tabucline[1pt]{1-6}
& & & & &\\[-0.3cm]
& \begin{tikzpicture}
\node[rotate=90] at (0,0) {\scriptsize one output};
\end{tikzpicture} & \begin{tikzpicture}[scale=0.9]
\draw[-stealth,line width=0.2mm] (0,1) -- (0,2);
\draw[-stealth, line width=0.2mm] (0,-2) -- (0,-1);
\draw[-stealth, line width=0.2mm, rotate=30] (0,-2) -- (0,-1);
\draw[-stealth, line width=0.2mm, rotate=-30] (0,-2) -- (0,-1);
\draw[fill=white, line width=0.2mm] (0,0) circle (0.8cm) node{\scriptsize $X$};
\end{tikzpicture}
&
\input{figures/morphismDAG1.tikz} & $P(X\, | \, ABC)$ &\\
& & & & &\\[-0.3cm]
\tabucline{2-5}
& & & & &\\[-0.3cm]
& \begin{tikzpicture}
\node[rotate=90] at (0,0) {\scriptsize ident. outputs}; 
\end{tikzpicture}
& \begin{tikzpicture}[scale=0.9]
\draw[-stealth,line width=0.2mm, rotate=15] (0,1) -- (0,2);
\draw[-stealth,line width=0.2mm, rotate=-15] (0,1) -- (0,2);
\draw[-stealth, line width=0.2mm] (0,-2) -- (0,-1);
\draw[-stealth, line width=0.2mm, rotate=30] (0,-2) -- (0,-1);
\draw[-stealth, line width=0.2mm, rotate=-30] (0,-2) -- (0,-1);
\draw[fill=white, line width=0.2mm] (0,0) circle (0.8cm) node{\scriptsize $X$};
\end{tikzpicture}
& \input{figures/morphismDAG2.tikz} & $P(XX \, | \, ABC)$ &\\
& & & & &\\[-0.3cm]
\tabucline{2-5}
& & & &\\[-0.3cm]
& \begin{tikzpicture}
\node[rotate=90] at (0,0) {\scriptsize diff. outputs}; 
\end{tikzpicture} & \vspace*{-0.35cm}\huge{\xmark} & \input{figures/morphismDAG3.tikz} & $P(XY \, | \, ABC)$\\[-0.3cm]
& & & & &\\[-0.5cm]
& \multirow{5}{*}{}\\
\end{tabu}
\end{center}
\medskip
Using string diagrams themselves as generalized causal models allows us to go beyond the DAG approach in several directions:
\begin{itemize}[label={--}]
	\setlength\itemsep{0.5em}
	\item String diagrams in Markov categories describe Markov kernels instead of just probability distributions. Therefore, the string diagram language allows for modeling causal structures with inputs, such as
	\begin{equation*} 
	\input{figures/tripartiteScenarioOpen.tikz}
	\end{equation*} 
	This describes a causal structure in which the input variable at the bottom does not have any particular distribution itself. For example, a conditional probability distribution is compatible with the above string diagram if
	$$P(X,A,Y \, | \, A) = P(X \, | \, A) \cdot P(Y \, | \, A)$$
	\item As indicated in the table, boxes in the string diagram can have more than a single output wire. This allows for causal structures like
	\begin{equation*}
	\input{figures/DAGImpossibleScenario.tikz}
	\end{equation*} 
	which are not native to the DAG framework (see \Cref{ex:dSepExamples}\ref{ex:dSepExamplesI} for a detailed discussion of this structure).
	In this situation a probability distribution $P(X, Y, Z_1, Z_2)$ is compatible with the given structure if
	$$P(X,Y,Z_1, Z_2) = P(X \, | \, Z_1) \cdot P(Y \, | \, Z_2) \cdot P(Z_1, Z_2).$$
	\item String diagrams allow for the use of identical boxes multiple times. This allows to represent for example the transition probabilities of a \emph{Markov decision processes} (see \cite[Section 2.1]{Pu94}) as the causal structure
	\begin{equation}\label{eq:MDP}
	\input{figures/MDP.tikz}
	\end{equation} 
	Here, $S_i$ denotes the state variable and $A_i$ the action variables at time $i$ and the boxes are all identical since such a decision process is given by a fixed conditional probability distribution
	$$P(S_{i+1} \, | \, S_{i}, A_{i}),$$
	which highlights that a Markov decision process has no additional internal memory besides the previous state.
Moreover, in this situation the types of the variables $A_1, \ldots, A_n$ must be the same, namely a particular action space $A$, and similarly the $S_1, \ldots, S_n$ all represent the same state space $S$. Therefore, the box $f$ can be seen as a morphism $S \otimes A \to S$.
\end{itemize}

Throughout the paper, we will usually follow the convention of labelling wires by their types together with an index that lets us address them individually, and at the same time identify these wires with the variables they carry. We do not index the boxes as there is no need to reference the distinct instances of the same type of box separately.

\subsection{Causal compatibility for Markov categories}
\label{ssec:CausalCompatibility}
A distribution is compatible with a causal model if it can be written as a composition in precisely the way specified by the corresponding string diagram. In other words, every type $W$ in the string diagram must be mapped to a concrete measurable space $FW$ and every box $f$ to a concrete Markov kernel $Ff$.

In category-theoretic language, this is captured in the following way \cite{Fo13}. A morphism $p$ in a concrete Markov category is compatible with a causal structure $\varphi$ if and only if there exists a Markov functor $F$ such that $p = F\varphi$. Intuitively, this functor acts in the following way:

\medskip
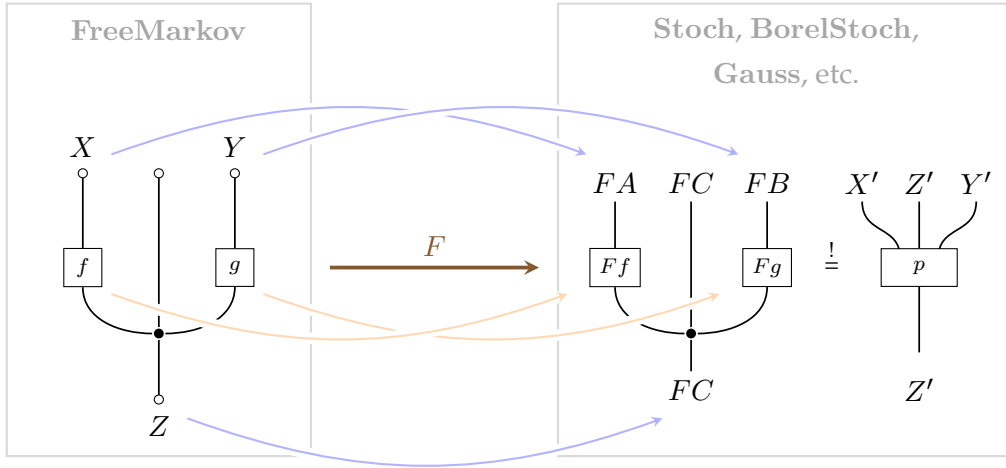
\begin{center} \begin{tikzpicture}

	\begin{pgfonlayer}{nodelayer}
	
	\draw[thick, gray!30!white] (-11,-4.5) rectangle (-3,7.5);
	\draw[thick, gray!30!white] (15.5,-4.5) rectangle (3.5,7.5);
	\node[style=none,gray!70!white] at (-7,6.8) {$\cat{FreeMarkov}$};
	\node[style=none,gray!70!white] at (9.5,6.8) {$\cat{Stoch}$, $\cat{BorelStoch}$,};
	\node[style=none,gray!70!white] at (9.5,5.6) {$\cat{Gauss}$, etc.};	
	
	\draw[thick, white, line width=0.15cm] (-6.25, -3.5) to[in=-160, out=-20] (6.25, -3.5);
	\draw[thick, blue, opacity=0.3, -stealth] (-6.25, -3.5) to[in=-160, out=-20] (6.25, -3.5);
	
	\draw[thick, white, line width=0.15cm] (-8.25, 3.5) to[in=160, out=20] (4.25, 3.5);
	\draw[thick, blue, opacity=0.3, -stealth] (-8.25, 3.5) to[in=160, out=20] (4.25, 3.5);
	
	\draw[thick, white, line width=0.15cm] (-4.25, 3.5) to[in=160, out=20] (8.25, 3.5);
	\draw[thick, blue, opacity=0.3, -stealth] (-4.25, 3.5) to[in=160, out=20] (8.25, 3.5);
	
	\draw[thick, white, line width=0.15cm] (-4.25, -0.2) to[in=-160, out=-20] (7.75,-0.2);
	\draw[thick, orange, opacity=0.3, -stealth] (-4.25, -0.2) to[in=-160, out=-20] (7.75,-0.2);
	
	\draw[thick, white, line width=0.15cm] (-8.25, -0.2) to[in=-160, out=-20] (3.75,-0.2);
	\draw[thick, orange, opacity=0.3, -stealth] (-8.25, -0.2) to[in=-160, out=-20] (3.75,-0.2);

		\node [style=boxSmall] (0) at (-9, 0.5) {$f$};
		\node [style=boxSmall] (1) at (-5, 0.5) {$g$};
		\node [style=blackDot] (2) at (-7, -1.25) {};
		\node [style=none] (3) at (-7, -3) {};
		\node [style=none] (4) at (-5, 3) {};
		\node [style=none] (5) at (-9, 3) {};
		\node [style=none] (6) at (-5, 3.7) {$Y$};
		\node [style=none] (7) at (-9, 3.7) {$X$};
		\node [style=none] (8) at (-7, -3.7) {$Z$};
		\node [style=whiteDot] at (-7, 3) {};
		\node [style=whiteDot] at (-7, -3) {};
		\node [style=whiteDot] at (-5, 3) {};
		\node [style=whiteDot] at (-9, 3) {};		
		
		\node [style=none] (9) at (-7, 3) {};
		\node [style=boxSmall] (11) at (5, 0.5) {$Ff$};
		\node [style=boxSmall] (12) at (9, 0.5) {$Fg$};
		\node [style=blackDot] (13) at (7, -1.25) {};
		\node [style=none] (14) at (7, -2.25) {};
		\node [style=none] (15) at (9, 2.25) {};
		\node [style=none] (16) at (5, 2.25) {};
		\node [style=none] (17) at (9, 2.75) {$FB$};
		\node [style=none] (18) at (5, 2.75) {$FA$};
		\node [style=none] (19) at (7, -2.75) {$FC$};
		\node [style=none] (20) at (7, 2.25) {};
		\node [style=none] (21) at (7, 2.75) {$FC$};
		\node [style = none] at (10.7,0.8) {$\stackrel{!}{=}$};
		\node [style=boxBroad] (22) at (13, 0.5) {$p$};
		\node [style=none] (23) at (13, -1.75) {};
		\node [style=none] (24) at (14.5, 2.25) {};
		\node [style=none] (25) at (13, 2.25) {};
		\node [style=none] (26) at (11.5, 2.25) {};
		\node [style=none] (27) at (13.5, 0.75) {};
		\node [style=none] (28) at (12.5, 0.75) {};
		\node [style=none] (29) at (13, -2.75) {$Z'$};
		\node [style=none] (30) at (11.5, 2.75) {$X'$};
		\node [style=none] (31) at (13, 2.75) {$Z'$};
		\node [style=none] (32) at (14.5, 2.75) {$Y'$};
	\end{pgfonlayer}
	\begin{pgfonlayer}{edgelayer}
		\draw [style=edge] (2) to (3.center);
		\draw [style=edge, in=-90, out=0] (2) to (1);
		\draw [style=edge, in=-90, out=180] (2) to (0);
		\draw [style=edge] (1) to (4.center);
		\draw [style=edge] (0) to (5.center);
		\draw [style=edge] (2) to (9.center);
		\draw [style=edge] (13) to (14.center);
		\draw [style=edge, in=-90, out=0] (13) to (12);
		\draw [style=edge, in=-90, out=180] (13) to (11);
		\draw [style=edge] (12) to (15.center);
		\draw [style=edge] (11) to (16.center);
		\draw [style=edge] (13) to (20.center);
		\draw [style=edge] (23.center) to (22);
		\draw [style=edge] (22) to (25.center);
		\draw [style=edge, in=-90, out=90, looseness=1.25] (27.center) to (24.center);
		\draw [style=edge, in=-90, out=90, looseness=1.25] (28.center) to (26.center);
	\end{pgfonlayer}

	\draw[line width=0.05cm, color=brown!70!black, -stealth] (-2.5, 0.5) -- node[midway, above]{$F$} (3, 0.5); 
\end{tikzpicture} \end{center}
\medskip
where $p$ is the given morphism in a concrete Markov category and $\cat{FreeMarkov}$ the free Markov category whose morphisms capture the causal models. Further, $X$ has type $A$, $Y$ has type $B$, and $Z$ has type $C$.
So if the original $p$ takes input from a measurable space $Z'$ and outputs values in measurable spaces $X'$, $Z'$ and $Y'$, then the types must match in the sense that $FA = X'$ etc.

\subsection{Categorical \texorpdfstring{$d$}{d}-separation}
\label{ssec:dsep}

The notion of $d$-separation for DAGs is a criterion relating conditional independence in causal models to the causal compatibility with a DAG. In \Cref{ssec:dseparation}, we introduce a categorical notion of $d$-separation from a different perspective. Although this notion looks different and \emph{conceptually} much simpler than classical $d$-separation, we prove that it coincides with the classical one when considering causal models on DAGs. 

An output wire $Z$ \emph{categorically $d$-separates} the output $X$ from output $Y$ if $X$ and $Y$ become disconnected upon \emph{marginalizing} over all wires not in $X,Y$ and $Z$ and then \emph{removing} the wire $Z$. We also express this by saying that $X$ is categorically $d$-separated from $Y$ by $Z$.

Consider, for example, the DAG in Equation \eqref{eq:DAGScenarioDAG}. $Z$ classically $d$-separates $X$ from $Y$, as one can see based on the fact that the only paths between $X$ and $Y$ are the collider $X \rightarrow W \leftarrow Y$ and the fork $X \leftarrow Z \rightarrow Y$. In the corresponding string diagram, Equation \eqref{eq:DAGScenario}, we witness categorical $d$-separation by first marginalizing over $W$, then removing the $Z$ wire, and finally observing that $X$ and $Y$ are disconnected, pictorially:
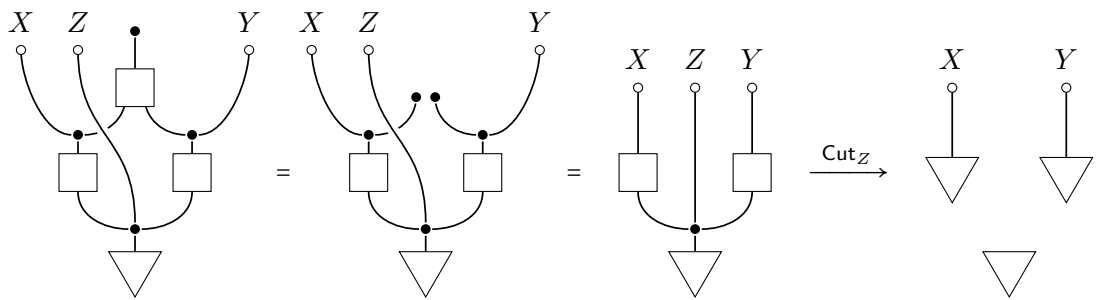
\begin{equation*} \input{figures/DAGscenario1MARG.tikz} \ = \input{figures/DAGscenario1MARG2.tikz} \ = \quad \input{figures/DAGscenario1XYZ.tikz} \quad \xrightarrow{\ \Cut_Z \ } \quad \input{figures/DAGscenario1CUT.tikz} \end{equation*} 
Here we have used the rules of marginalization, namely Equation \eqref{eq:delBox} and Equation \eqref{eq:delCopy}.

On the other hand, $X$ is not $d$-separated from $Y$ by $W$ and $Z$ due to the collider $X \rightarrow W \leftarrow Y$. In the string diagram, this is apparent since upon removing the wires $Z$ and $W$,
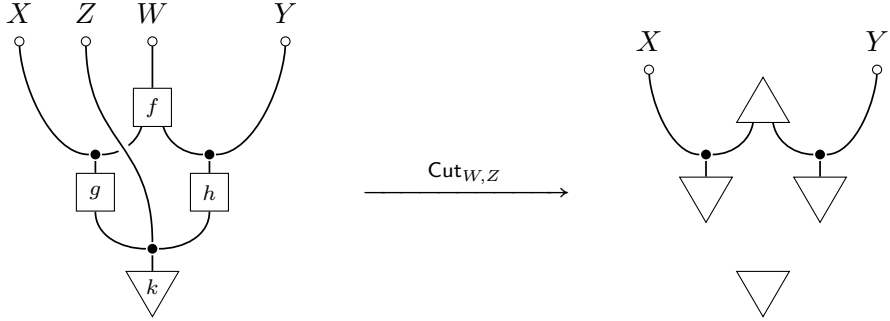
\begin{equation*} \input{figures/DAGscenario1WithoutType.tikz} \qquad \xrightarrow{\qquad \Cut_{W,Z} \qquad} \qquad \input{figures/DAGscenario1CUT2.tikz} \end{equation*} 
$X$ and $Y$ are still connected. 

We note that there also are approaches to classical $d$-separation based on topological disconnectedness, such as using the moralized graph of a DAG \cite[Proposition 3.25]{La96}. An application of this criterion consists of four steps: first, \emph{marginalizing} a certain subset of variables not present in the $d$-separation; second, \emph{moralizing} the DAG to an undirected graph; third, removing every vertex associated to $Z$; and finally, checking \emph{disconnectedness}.
Categorical $d$-separation is conceptually simpler than the DAG approach since it consists only of three steps, namely marginalizing \emph{all} variables not involved in the putative $d$-separation, removing the $Z$ wires, and then checking disconnectedness. The moralization step is omitted since its function is already encoded in the structure of string diagrams.
 Therefore, categorical $d$-separation is a more intuitive procedure compared to the standard definitions of $d$-separation on DAGs.

\subsection{$d$-separation and causal compatibility}
\label{ssec:dsepCrit}

For Bayesian networks, $d$-separation detects conditional independences for any compatible probability distribution. In particular, if two output wires $X$ and $Y$ are $d$-separated by the output wire $Z$, we have that $X$ is independent of $Y$ given $Z$ (denoted as $X \perp Y \ | \ Z$).

Conditional independence in Markov categories is defined string diagrammatically. A probability distribution $p$ on a product space $X \times Y \times Z$ displays the conditional independence $X \perp Y \ | \ Z$ if it can be written in the form
\begin{equation*}
\input{figures/tripartiteProbability.tikz} \quad = \quad \input{figures/conditionalIndepTripartiteProbability.tikz}
\end{equation*} 
This reflects the classical notion of conditional independence with the two unnamed boxes corresponding to the conditionals\footnote{See also \cite[Remark 12.4]{Fr20} for discussion.}; in the context of discrete probability, it encodes the equation
\begin{equation*} P(X=x,Y=y,Z=z) \,=\, P(X = x \, | \, Z = z) \cdot P(Y = y \, | \, Z = z) \cdot P(Z = z).
\end{equation*} 
But also for Gaussian random variables as well as measures on standard Borel spaces, one recovers the intuitive notions of conditional independence \cite[Section 12]{Fr20}. Moreover, the diagrammatic definition is even sufficient to derive the semigraphoid properties \cite[Lemma 12.5]{Fr20}.
It also generalizes to a notion of conditional independence for morphisms with inputs (see \Cref{def:condInd_diff}).

In \Cref{ssec:causalCompatibility}, we prove that the categorical $d$-separation criterion applies to generalized causal models in Markov categories. For this reason, we define a notion of conditional independence which applies to arbitrary Markov kernels (\Cref{def:condInd_diff}). We first prove the soundness of the $d$-separation, namely, if $Z$ categorically $d$-separates $X$ from $Y$, then $X \perp Y \ | \ Z$ (\Cref{cor:dseparationCondIndependence}). Second, we prove the completeness of $d$-separation for causal compatibility: if a Markov kernel satisfies the global Markov property for a causal structure (i.e.\ every $d$-separated triple shows conditional independence), then the Markov kernel is compatible with the structure (see \Cref{thm:causalCompdSep} for the precise assumptions and also the equivalence with the local Markov property).

A central assumption for the proof is the existence of conditionals (see \Cref{def:existenceConditionals}). Intuitively it says that the outputs of any morphism $f$ can be generated successively while having access to all prior information. Conditionals exist in discrete probability, measure-theoretic probability on standard Borel spaces and in Gaussian probability, and this facilitates the application of our results to all of these cases, where the second case includes continuous random variables.

\section{gs-monoidal and Markov categories}
\label{sec:MarkovCat}

In the following, we recall the definitions of Markov categories and gs-monoidal categories. Markov categories are the basic structure modeling different flavors of probability, including discrete random variables, Gaussian random variables, continuous random variables on standard Borel spaces, or random variables on arbitrary measurable spaces.
We assume some familiarity with symmetric monoidal categories and string diagrams. We refer to \cite{Le16, Pe19} for a general introduction to category theory and to \cite{Ba08, Co09,He19} for a more detailed account of monoidal categories and their string diagrammatic calculus.

Intuitively, a monoidal category is a category with a product structure $\otimes$. Specifically, two objects $A,B$ in the category have an associated product object $A \otimes B$. Moreover, if $f: A \to X$ and $g: B \to Y$ are morphisms, then $\otimes$ gives rise to a new morphism $f \otimes g: A \otimes B \to X \otimes Y$. This monoidal product can have different realizations. For example, the category $\cat{Set}$ is monoidal with the Cartesian product of sets and functions. A monoidal category is \emph{symmetric} if the objects $A \otimes B$ and $B \otimes A$ are naturally isomorphic in a nice way. All of the monoidal categories in this paper are symmetric.

The notion of gs-monoidal categories goes back to Gaducci's thesis \cite[Definition 3.9]{Ga96} and and an associated paper by Corradini and Gadducci \cite{Co99}. There, it was considered with a different motivation in the context of term graphs and term graph rewriting.

\begin{definition}[{gs-monoidal category and Markov category}]
\phantom{a}
\label{def:MarkovCat}
\begin{enumerate}[label=(\roman*)] 
	\item A \emph{gs-monoidal category} $\cat{C}$ is a symmetric monoidal category with monoidal unit $I$ equipped with a commutative comonoid structure for every object $X \in \cat{C}$ given by a counit $\Del{X} \ : \ X \to I$ and a comultiplication $\Copy{X} \ : \ X \to X \otimes X$. In the string diagrammatic notation, these operations are depicted as
\[ \Del{X} \ \coloneqq \  \begin{tikzpicture}
	\begin{pgfonlayer}{nodelayer}
		\node [style=blackDot] (0) at (0, 1) {};
		\node [style=none] (3) at (0, -1) {};
	\end{pgfonlayer}
	\begin{pgfonlayer}{edgelayer}
		\draw [style=edge] (3.center) to (0);
	\end{pgfonlayer}
\end{tikzpicture}
 \hspace{3cm} \Copy{X} \ \coloneqq \ \input{figures/copy.tikz}\]
They are required to satisfy the commutative comonoid equations, diagrammatically given by
\begin{equation}
	\label{eq:commcom}
	\begin{gathered}
		\input{figures/comonoid1.tikz} \ = \ \input{figures/comonoid2.tikz} \hspace{1.4cm} \input{figures/leftCopyContraction.tikz} \ = \ \input{figures/rightCopyContraction.tikz} \ = \ \begin{tikzpicture}
	\begin{pgfonlayer}{nodelayer}
		\node [style=none] (0) at (0, -1) {};
		\node [style=none] (1) at (0, 2) {};
	\end{pgfonlayer}
	\begin{pgfonlayer}{edgelayer}
		\draw [style=edge] (0.center) to (1.center);
	\end{pgfonlayer}
\end{tikzpicture}
 \hspace{1.4cm} \input{figures/copy_long.tikz} \ = \ \input{figures/copy_swap.tikz}
	\end{gathered}
\end{equation} 
and to be compatible with the monoidal structure, i.e.\
\begin{equation}
	\label{eq:comonoidalStruc}
		\begin{tikzpicture}
	\begin{pgfonlayer}{nodelayer}
		\node [style=blackDot] (0) at (0, 1) {};
		\node [style=none] (3) at (0, -1) {};
		\node [style=none] (4) at (0, -1.5) {$A \otimes B$};
	\end{pgfonlayer}
	\begin{pgfonlayer}{edgelayer}
		\draw [style=edge] (3.center) to (0);
	\end{pgfonlayer}
\end{tikzpicture}
 \ = \ \ \ \input{figures/delTensor2.tikz} \hspace{1.5cm} \input{figures/copy_tensor1.tikz} \ \ \ = \  \input{figures/copy_tensor2.tikz}
\end{equation} 
as well as
\begin{equation} 
\label{eq:copyDelI} \begin{tikzpicture}
	\begin{pgfonlayer}{nodelayer}
		\node [style=blackDot] (0) at (0, 1) {};
		\node [style=none] (3) at (0, -1) {};
		\node [style=none] (4) at (0, -1.5) {$I$};
	\end{pgfonlayer}
	\begin{pgfonlayer}{edgelayer}
		\draw [style=edge] (3.center) to (0);
	\end{pgfonlayer}
\end{tikzpicture}
 \quad = \quad \input{EmptyBox.tikz} \hspace{2cm} \input{figures/emptyCopy.tikz} \quad = \quad \input{EmptyBox.tikz} \hspace{0.4cm}
\end{equation} 
	\item A gs-monoidal category is called a \emph{Markov category} if $\Del{}$ is in addition natural, i.e.\ if for all morphisms $f$,
\begin{equation} 
\label{eq:naturality} \begin{tikzpicture}
	\begin{pgfonlayer}{nodelayer}
		\node [style=boxSquare] (2) at (0, 0) {$f$};
		\node [style=blackDot] (4) at (0, 1.5) {};
		\node [style=none] (5) at (0, -1.5) {};
	\end{pgfonlayer}
	\begin{pgfonlayer}{edgelayer}
		\draw [style=edge] (5.center) to (4);
	\end{pgfonlayer}
\end{tikzpicture}
 \quad = \quad  
\end{equation} 
\end{enumerate}
\end{definition}

We refer to~\cite[Remark~2.2]{Fr22} for more details on the multifarious history of these notions.
When considering morphisms with multiple inputs or outputs, we often denote the collective inputs and outputs as (multi)sets instead of tensor products. For example, we write
\begin{equation*}
 \input{figures/MorphismSet1.tikz} \ = \quad \input{figures/MorphismSet2.tikz}
\end{equation*} 
where $\mathcal{V} = \{A,B,C\}$ and $\mathcal{W} = \{X, Y\}$.
Modulo some abuse of notation, the order of the inputs and outputs is irrelevant since we can always permute the wires, and therefore it is enough to consider $\mathcal{V}$ and $\mathcal{W}$ as multisets rather than lists.

\begin{example}
The most important examples of Markov categories are the following:
\begin{enumerate}[label=(\alph*)]
	\item $\cat{FinStoch}$ is the Markov category of finite stochastic maps. The objects are finite sets $X$, morphisms $p: I \to X$ are probability distributions and general morphisms $f: X \to Y$ are stochastic matrices with entries $f(y\, | \, x)$. The symmetric monoidal structure is given by the Cartesian product of sets.
	\item $\cat{Stoch}$ is the Markov category of arbitrary Markov kernels on measurable spaces. The objects are measurable spaces, morphisms $p: I \to X$ are probability measures, and general morphisms $f : X \to Y$ are measurable Markov kernels.
	\item $\cat{BorelStoch}$ is given by $\cat{Stoch}$ restricted to standard Borel spaces as objects.
	\item $\cat{Gauss}$ is the Markov category of Gaussian probability distributions. The objects are the spaces $\mathbb{R}^n$, morphisms $p: I \to \mathbb{R}^n$ are Gaussian probability measures and general morphisms $f: \mathbb{R}^n \to \mathbb{R}^m$ can be understood as stochastic maps of the form
	\begin{equation*}
	 x \mapsto A x + \xi 
	\end{equation*} 
	where $A$ is any real $m \times n$ matrix and $\xi$ is an independent Gaussian variable with given mean and variance.
	$\cat{Gauss}$ is a (non-full) subcategory of $\cat{BorelStoch}$.
	\item $\cat{FinSet}$ is the Markov category of finite sets and (deterministic) functions. Therefore, a morphism $f: X \to Y$ is a function from $X$ to $Y$ and $p: I \to X$ is a single element of $X$. The copy maps are given by
	\begin{align*}
		\Copy{X} : {} 	& X \to X \times X \\
				& x \mapsto (x,x)
	\end{align*}
	and the deletion morphism $\Del{X}: X \to I$ is the only map to the one-element set $I$.
	
	Equivalently, a morphism $f: X \to Y$ can be thought of a matrix $(f(y \, | \, x))_{y \in Y, x \in X}$ with entries in $\{0,1\}$ by defining $f(y\, | \, x) = 1$ if and only if $f(x) = y$.
	This shows that there is a functor $\cat{FinSet} \to \cat{FinStoch}$ through which the Markov category structure of $\cat{FinSet}$ is inherited from $\cat{FinStoch}$.
	$\cat{FinSet}$ also generalizes to the category $\cat{Set}$ by considering arbitrary sets.
	\item $\cat{FinSetMulti}$ is the Markov category of finite sets and multivalued maps. A morphism $f:X \to Y$ is given by a matrix $(f(y \, | \, x))_{y \in Y, x \in X}$ with entries in $\{0,1\}$ and subject to the condition that for every $x$, we have $f(y\, | \, x) = 1$ for some $y$
	If $f(y \, | \, x) = 1$, then $y$ is possible when given $x$, while $f(y \, | \, x) = 0$ implies that $y$ is impossible given $x$. We define composition of two morphisms $f: X \to Y$ and $g:Y \to Z$ via
	$$(g \circ f)(z\, | \, x) \coloneqq \sum_{y \in Y} g(z \, | \, y) \cdot f(y \, | \, x)$$
	where we use the convention that $1+1 = 1$. This suggests the same notation as for morphisms as in $\cat{FinStoch}$.
\end{enumerate}
\end{example}

For further details on the composition, the symmetric monoidal structure and the Markov category structure of these examples, we refer to \cite{Fr20}.

In our context, gs-monoidal categories that are not Markov categories play more of an auxiliary role which we will detail below.

\section{Free Markov categories and generalized causal models}
\label{sec:FreeMarkov}

Causal models are a framework for studying and modeling dependencies between random variables. 
In this section, we introduce such a framework in the language of Markov categories. We therefore investigate causal relationships independently of the particular notion of probability behind it (discrete, measure-theoretic, Gaussian, etc).

Free Markov categories are the tailored notion for these purposes. These categories contain precisely all blueprints for causal networks that can be built from a bunch of given causal mechanisms. A morphism in this category is then a \emph{generalized causal model}. 

The remainder of this section explains this in technical detail based on the formalism of free gs-monoidal categories and free Markov categories from \cite{Fr22}.\footnote{Another very similar construction of free gs-monoidal categories has been given independently in \cite{Mi22}.} This part is structured as follows: In \Cref{sec:hyp}, we introduce the category of hypergraphs. In \Cref{sec:stringDiagrams} we use hypergraphs to define gs-monoidal string diagrams, free gs-monoidal categories, and subsequently free Markov categories. In \Cref{sec:compatibility} we introduce generalized causal models and causal compatibility of morphisms in arbitrary Markov categories.

\subsection{The category of hypergraphs}
\label{sec:hyp}

A gs-monoidal string diagram, and therefore also a generalized causal model is defined as a hypergraph with extra structure.
To define the relevant notion of hypergraph following~\cite{Bo16}, let first $\cat{I}$ be the category defined in the following way:
\begin{enumerate}[label=(\roman*)]
	\item The set of objects is given by $\{(k, \ell) \mid k, \ell \in \mathbb{N}\} \cup \{\star\}$.
	\item Besides the identity morphisms, for every $(k,\ell)$ there are $k + \ell$ different morphisms
		$$\mathsf{in}_1, \ldots, \mathsf{in}_k, \mathsf{out}_1, \ldots, \mathsf{out}_{\ell}: (k,\ell) \to \star,$$
		and no other morphisms.
\end{enumerate}

It is not necessary to specify a composition rule in $\cat{I}$ since no compositions except the trivial ones can be formed.

\begin{definition}
A functor $G: \cat{I} \to \cat{Set}$ is called a \emph{hypergraph}. Accordingly, we define the functor category
\begin{equation*}
\cat{Hyp} \coloneqq \cat{Set}^{\cat{I}} 
\end{equation*} 
to be the category of hypergraphs.
\end{definition}

Intuitively, the functor $G$ characterizes our common interpretation of (directed) hypergraphs in the following way:
\begin{enumerate}[label=(\roman*)]
	\item $W(G) \coloneqq G(\star)$ is the set of vertices, which we will call \emph{wires}.
	\item $B_{k,\ell}(G) \coloneqq G((k,\ell))$ is the set of hyperedges, which we will call \emph{boxes}, with $k$ inputs and $\ell$ outputs.
	\item $G(\mathsf{in}_i)$ specifies the $i$th input wire of every box.
	\item $G(\mathsf{out}_j)$ specifies the $j$th output wire of every box.
\end{enumerate}

While the set of wires and boxes of a hypergraph may be infinite, the number of inputs and outputs of a box is always finite. We present a pictorial representation of a hypergraph in \Cref{fig:hypergraph}.

\begin{figure}[h!]
\begin{center}
\input{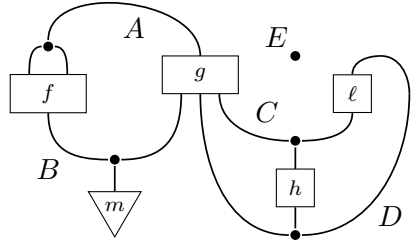}
\end{center}
\caption{Pictorial representation of a hypergraph with wire set $\{A,B,C,D,E\}$ and box set $\{f,g,h,m,n\}$. For example, the box $f$ has one input incident to the wire $B$ and two outputs both incident to the wire $A$. The wire $E$ is not incident to any box.}
\label{fig:hypergraph}
\end{figure}

Given a box $b \in B_{k,\ell}(G)$ and a wire $A \in W(G)$, we define
\begin{align*}
	\mathsf{in}(b,A) &\coloneqq \big|\big\{j \in \{1, \ldots, \ell\}: \mathsf{in}_{j}(b) = A\big\}\big|, \\
	\mathsf{out}(b,A) &\coloneqq \big|\big\{i \in \{1, \ldots, k\}: \mathsf{out}_{i}(b) = A\big\}\big|.
\end{align*} 
Thus $\mathsf{in}(b,A)$ and $\mathsf{out}(b,A)$ counts the number of incoming or outgoing wires of type $A$ in the box $b$. We also define the sets of inputs and outputs as
\begin{align*}
	\In(b) &\coloneqq \big\{\In_i(b): i \in \{1, \ldots, \ell\}\big\}, \\
	\Out(b) &\coloneqq \big\{\Out_i(b): i \in \{1, \ldots, k\}\big\},
\end{align*} 
where repeated wires are counted only once.

Next, we analyze the morphisms in $\cat{Hyp}$. Since $\cat{Hyp}$ is a functor category, a morphism $\alpha: F \to G$ is precisely a natural transformation $\alpha: F \Rightarrow G$. Such a natural transformation is fully determined by its components
\begin{equation*} 
\alpha_{\star}: W(F) \to W(G) \quad \text{ and } \quad \alpha_{(k,l)}: B_{(k,\ell)}(F) \to B_{(k,\ell)}(G) \quad \text{ for all } k, \ell \in \mathbb{N} 
\end{equation*} 
satisfying naturality, i.e.\ commutativity of all diagrams of the form
\begin{equation*}
\begin{tikzcd}
B_{k,\ell}(F) \ar{d}{\alpha_{k,\ell}} \ar{r}{\mathsf{in}_i} & W(F) \ar{d}{\alpha_{k,\ell}} \\
B_{k,\ell}(G) \ar{r}{\mathsf{in}_i} & W(G)
\end{tikzcd} \qquad \qquad \qquad \begin{tikzcd}
B_{k,\ell}(F) \ar{d}{\alpha_{k,\ell}} \ar{r}{\mathsf{out}_j} & W(F) \ar{d}{\alpha_{k,\ell}} \\
B_{k,\ell}(G) \ar{r}{\mathsf{out}_j} & W(G).
\end{tikzcd}
\end{equation*} 
In other words, every natural transformation is a structure-preserving map between hypergraphs, i.e.~if box $f$ is incident to wire $A$ in its $i$th input in the hypergraph $F$, then the same applies to their images with respect to $\alpha$ in the hypergraph $G$.

A hypergraph can contain an infinite set of wires and boxes. In the following we mainly restrict to \emph{finite hypergraphs}, i.e.\ functors $G : \cat{I} \to \cat{Set}$ where $W(G)$ and
$$ B(G) \coloneqq \coprod_{k,\ell \in \mathbb{N}} B_{k,\ell}(G).$$
are finite sets.
We denote the corresponding full subcategory of $\cat{Hyp}$ by $\cat{FinHyp}$.

\subsection{gs-monoidal string diagrams and free Markov categories}
\label{sec:stringDiagrams}

The pictorial representation of hypergraphs indicates their use for modeling causal structures in a categorical framework. In this subsection, we construct the free Markov category associated to a fixed hypergraph $\Sigma$. This is a Markov category in which the morphisms are string diagrams formed out of the boxes in $\Sigma$, and we argue that these morphisms can be used as generalized causal models.

However, three apparent problems make hypergraphs not directly applicable to represent string diagrams:
\begin{enumerate}
	\item General hypergraphs might contain loops, for example in the sense that an output wire of a box may directly feed back as an input.
	\item While the splitting of a wire into two represents the copying of values and makes sense in any Markov category, the merging of wires as in \Cref{fig:hypergraph} does not make sense.
	\item A hypergraph in itself does not include any information about which wires are inputs or outputs of the overall diagram.
\end{enumerate}	
We resolve these issues by restricting to acyclic and left monogamous hypergraphs and by representing gs-monoidal string diagrams by cospans thereof:

\begin{definition}
\label{def:stringDiagram}
Let $\Sigma$ be a hypergraph. A \emph{gs-monoidal string diagram} for $\Sigma$ is a cospan in the slice category $\cat{FinHyp}/\Sigma$ of the form
\begin{equation*}
\begin{tikzcd} 
& G & \\
\underline{n} \arrow{ru}{p}  & & \underline{m} \arrow[swap]{lu}{q}
\end{tikzcd}
\end{equation*} 
satisfying that:
\begin{enumerate}[label=(\roman*)]
	\item $G$ is \emph{acyclic}, i.e.\
	there is \emph{no} family of wires $A_0, \ldots, A_{k-1} \in W(G)$ and boxes $f_0, \ldots, f_{k-1} \in B(G)$ such that
	\begin{equation*}
		\mathsf{in}(f_i, A_i) \geq 1 \quad \text{ and } \quad \mathsf{out}(f_i, A_{i+1}) \geq 1,
	 \end{equation*} 
	where the subscripts are modulo $k$.
	\item $G$ is \emph{left monogamous}, i.e.\ for every wire $W \in W(G)$ we have
		\begin{equation*}
		 |p^{-1}(W)| + \sum_{f \in B(G)} \mathsf{out}(f,W) = 1.
		 \end{equation*}  
\end{enumerate}
By abuse of notation, we also write $G$ for the underlying hypergraph of the object $G$ in $\cat{FinHyp} / \Sigma$, and we write $\mathsf{type} : G \to \Sigma$ for the morphism that makes it into an object of $\cat{FinHyp} / \Sigma$.
\end{definition}

In this definition, the \emph{discrete hypergraph} $\underline{n}$ is defined to have $W(\underline{n}) = \{1,\ldots,n\}$ and contains no boxes. Thus the morphism $p : \underline{n} \to G$ simply equips some wires with labels from $1,\ldots,n$, thereby telling us which wires of $G$ are input wires of the overall diagram and in which order. The other cospan leg $q : \underline{m} \to G$ similarly encodes the $m$ output wires.

\begin{example}
\label{ex:ConstructionGsMonoidal}
	Consider the two hypergraphs
	$$ \Sigma \ = \ \input{figures/Sigma1.tikz}$$
	and
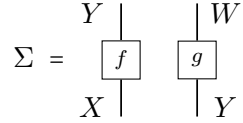
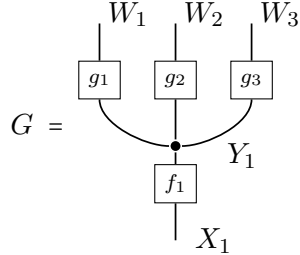
	$$ G \ = \ \input{figures/G1.tikz}$$
	Note that we represent $\Sigma$ by two distinct boxes instead of connecting them via the $Y$ wire in order to highlight the purpose of $\Sigma$ being the collection of elementary building blocks.
	The hypergraph $G$ is an object of the slice category $\cat{FinHyp}/\Sigma$ through the hypergraph morphism $\mathsf{type}: G \to \Sigma$ that maps the wires $X_1$ to $X$, $Y_1$ to $Y$, and $W_1, W_2, W_3$ to $W$, and that maps the box $f_1$ to $f$, and the boxes $g_1, g_2,g_3$ to $g$. 
	
	We define a gs-monoidal string diagram as a cospan by specifying two discrete hypergraphs $\underline{n}$ and $\underline{m}$ together with mappings $p: \underline{n} \to W(G)$, $q: \underline{m} \to W(G)$ that specify which wires are connected to the input and to the output output. Let us consider two different choices:
	\begin{enumerate}[label=(\roman*)]
		\item If $\underline{n} = \{1,2\}$ and $\underline{m} = \{1,2,3,4,5\}$ as well as
			\begin{align*}
				p \: & : \: 1 \mapsto X_1, \: 2 \mapsto X_1, \\[2pt]
				q \: & : \: 1 \mapsto W_1, \: 2 \mapsto W_2, \: 3 \mapsto Y_1, \: 4 \mapsto Y_1, \: 5 \mapsto W_3,
			\end{align*}
			then we get a cospan that corresponds to a diagram that could be drawn as
		$$ \input{figures/G3.tikz} $$
		\item If $\underline{n} = \{1\}$ and $\underline{m} = \{1,2,3\}$ as well as
			\begin{align*}
				p \: & : \: 1 \to X_1, \: \\[2pt]
				q \: & : \: i \mapsto W_i,
			\end{align*}
			then we get a cospan that corresponds to the gs-monoidal string diagram drawn as
		$$ \input{figures/G4.tikz} $$
	\end{enumerate}
	Note that we have secretly used the convention to draw the discrete hypergraph $\underline{n}$ (representing the input) on the bottom and $\underline{m}$ (representing the output) on the top. Moreover, we used the convention that wires at the bottom of a box are input wires and wires on the top are output-wires.
\end{example}

\begin{remark}
	Throughout the rest of the paper, we will omit the numbering of input and output based on the convention that inputs and outputs are numbered from left to right. Moreover, we will only label the boxes by their type, highlighting that the causal mechanisms are identical. For example, we will draw the string diagram \Cref{ex:ConstructionGsMonoidal} (ii) in simplified form as follows:
	$$ \input{figures/G5.tikz} $$
	Moreover, this string diagram is also gs-monoidal since it is acyclic and left monogamous. In contrast, the string diagram in \Cref{ex:ConstructionGsMonoidal} (i) is not left monogamous since the wire $X_1$ does not arise from a unique global input. 
	
\end{remark}

For further examples and non-examples of gs-monoidal string diagrams, we refer to \Cref{fig:hypergraph_morphism_examples}. 

Pictorially, an acyclic hypergraph does not contain a family of wires which form a loop. Further, left monogamy requires that every wire in the hypergraph arises as either a global input or as an output of a box in precisely one way, ensuring that no ``merging'' of wires occurs. See \Cref{fig:hypergraph_morphism_examples} for an illustration of all of this.

\begin{figure}
\centering
\begin{subfigure}[b]{0.28\textwidth}
\begin{equation*} 
\input{figures/hypergraph_example_morphism1.tikz} 
\end{equation*}
\subcaption{}
\label{stringdiagram1}
\end{subfigure}\hspace{1cm}
\begin{subfigure}[b]{0.28\textwidth}
\begin{equation*}
\input{figures/hypergraph_example_morphism2.tikz} 
\end{equation*} 
\subcaption{}
\label{stringdiagram2}
\end{subfigure}\hspace{1cm}
\begin{subfigure}[b]{0.28\textwidth}
\begin{equation*}
\input{figures/hypergraph_example_morphism3.tikz}
\end{equation*} 
\caption{}
\label{stringdiagram3}
\end{subfigure}
\caption{One example (\subref{stringdiagram1}) and two non-examples (\subref{stringdiagram2},\hspace{2pt}\subref{stringdiagram3}) of gs-monoidal string diagrams, where $\Sigma$ is the hypergraph from \Cref{fig:hypergraph}. The hypergraph in (\subref{stringdiagram2}) is not left monogamous and (\subref{stringdiagram3}) is not acyclic.
}
\label{fig:hypergraph_morphism_examples}
\end{figure}

The notion of gs-monoidal string diagram is the main ingredient for constructing a gs-monoidal category whose morphisms are freely generated by the wires and boxes in a fixed hypergraph $\Sigma$.
Indeed we can now define the category $\cat{FreeGS}_{\Sigma}$ as follows:
\begin{enumerate}[label=(\roman*)]
	\item Objects are all hypergraph morphisms $\sigma: \underline{n} \to \Sigma$ for $n \in \mathbb{N}$, or equivalently finite sequences of wires in $\Sigma$.
	\item Morphisms are the isomorphism classes of gs-monoidal string diagrams.
\end{enumerate}
Composition in $\cat{FreeGS}_\Sigma$ is defined by a pushout which coincides with the way of stacking and connecting up drawings of string diagrams. The gs-monoidal structure likewise corresponds to the obvious operations in terms of string diagrams. We refer to~\cite{Fr22} for details.
In the following, we will not distinguish between a gs-monoidal string diagram and its isomorphism class.

Although $\cat{FreeGS}_\Sigma$ is a gs-monoidal category, it is typically not a Markov category. For example, the first step in the following simplification does not hold since the (cospans of) hypergraphs are not isomorphic, while the second equation does hold:
\begin{equation} 
\label{eq:eliminablemorphism} 
\input{figures/elim_boxes1.tikz} \quad \neq \quad \input{figures/elim_boxes2.tikz} \quad = \quad \input{figures/elim_boxes3.tikz}
\end{equation}
In the following, we define the free Markov category $\cat{FreeMarkov}_{\Sigma}$ by taking a quotient of $\cat{FreeGS}_{\Sigma}$ which enforces Equation \eqref{eq:naturality}, so that also the first equation above becomes true.

\begin{definition}
Let
\begin{equation}
\label{eliminable}
\varphi \coloneqq \begin{tikzcd}
& G & \\
\underline{n} \arrow{ru}{p}  & & \underline{m} \arrow[swap]{lu}{q}
\end{tikzcd}
\end{equation} 
be a gs-monoidal string diagram.
\begin{enumerate}[label=(\roman*)]
	\item A box $b \in B(G)$ is called \emph{eliminable} if each output of $b$ gets discarded, i.e.\ if for every $W \in W(G)$ such that $\mathsf{out}(b,W) > 0$ we have
	\begin{enumerate}[label=(\alph*)]
		\item $q^{-1}(W) = \emptyset$.
		\item $\mathsf{in}(b', W) = 0$ for every box $b' \in B(G)$.
	\end{enumerate}
	\item $\varphi$ is called \emph{normalized} if it contains no eliminable boxes.
\end{enumerate}
\end{definition}

Every gs-monoidal string diagram has a normalized version obtained by iteratively applying the rule of Equation \eqref{eq:naturality} to any eliminable box. Since every diagram is finite, this procedure terminates after finitely many steps, and we reach the normalized version. In addition, this diagram is unique since the order of elimination does not matter.

The free Markov category $\cat{FreeMarkov}_\Sigma$ is now defined just as $\cat{FreeGS}_\Sigma$, but with morphisms restricted to the normalized gs-monoidal string diagrams. 
The composition of morphisms is then defined as composition in $\cat{FreeGS}_\Sigma$ followed by normalization since the composition of two normalized diagrams need not be normalized.
See~\cite{Fr22} for details. In particular, any normalized gs-monoidal string diagram $\varphi$ as in \Cref{eliminable} is a morphism of the form
\begin{equation}
\label{phi_morphism}
	\varphi \: : \: \bigotimes_{i=1}^n \mathsf{type}(p(i)) \longrightarrow \bigotimes_{j=1}^m \mathsf{type}(q(j))
\end{equation}
in $\cat{FreeMarkov}_\Sigma$.

\begin{example}
The morphism 
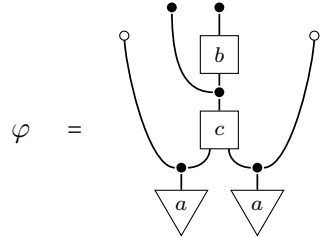
\begin{equation*} \varphi \quad = \quad \input{figures/NonConectednessExample2.tikz} \end{equation*} 
is not normalized since the output of $b$ gets discarded. Applying Equation \eqref{eq:naturality}, also the output of $c$ get discarded. Therefore the normalization of $\varphi$ is
\medskip
\begin{equation*}
	\norm(\varphi) \quad = \quad	\input{figures/NonConectednessExample3.tikz}
\end{equation*}  
\end{example}
In general, normalizing a gs-monoidal string diagram defines a strict gs-monoidal functor
\[
	\norm : \cat{FreeGS}_\Sigma \to \cat{FreeMarkov}_\Sigma
\]
that is identity-on-objects.

\subsection{Causal models and causal compatibility}
\label{sec:compatibility}

We now introduce the notion of a generalized causal model and define when a morphism in a Markov category is considered compatible with a generalized causal model.

\begin{definition}[{\cite[Definition~7.1]{Fr22}}]
Given a hypergraph $\Sigma$, a \emph{generalized causal model} is a normalized gs-monoidal string diagram (represented by \eqref{eliminable}) over $\Sigma$ such that $q$ is injective.
\end{definition}
Intuitively, a generalized causal model is a morphism in $\cat{FreeMarkov}_{\Sigma}$ where the injectivity of $q$ ensures that each wire is connected to at most one output. This lets us identify the global inputs and outputs with the wires in $W(G)$ (see \Cref{not:CausalModel}). In the traditional terminology of random variables, the injectivity of $q$ guarantees that different outputs correspond to different variables. \Cref{fig:ExCausalModels} shows examples of generalized causal models.

One relevant subclass of generalized causal models are pure blooms. These morphisms represent causal models in which all variables are observed, i.e.\ every wire is an output in exactly one way, such as in \Cref{fig:ExCausalModels}(c).

\begin{definition}[\cite{Fr22}]
\label{def:purebloom}
Let $\varphi$ be a generalized causal model represented by a gs-monoidal string diagram
\begin{equation*}
\varphi = \begin{tikzcd}
& G & \\
\underline{n} \arrow{ru}{p}  & & \underline{m} \arrow[swap]{lu}{q}
\end{tikzcd}
\end{equation*} 
Then $\varphi$ is called \emph{pure bloom} if $q$ is a bijection on wires.
\end{definition}

\begin{remark}
\label{rem:DAGRep}
Bayesian networks are a strict subclass of generalized causal models. More precisely, a generalized causal model $\varphi$ arises from a DAG if and only if
\begin{enumerate}[label=(\roman*)]
	\item $\varphi$ is pure bloom (no latent variables),
	\item $\In(\varphi) = \emptyset$ (no global inputs), and
	\item $|\Out(b)| = 1$ for every box $b \in B(G)$. 
\end{enumerate}
\end{remark}

\begin{figure}[h!]
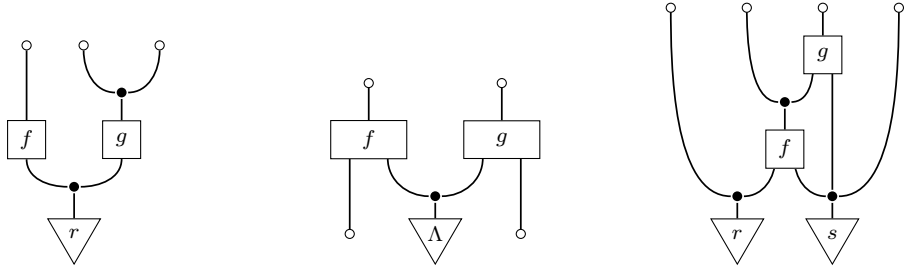

\centering
	\subcaptionbox{Non-example of a generalized causal model.}[.3\linewidth]{\input{figures/NoCausalModel.tikz}}
	\subcaptionbox{The Bell scenario.}[.3\linewidth]{\input{figures/BellScenarioMorphism.tikz}}
	\subcaptionbox{The instrumental scenario with every variable being observed.}[.3\linewidth]{\input{figures/InstrumentalScenarioNoVar.tikz}}
	\caption{(Non-)Examples of generalized causal models. While (b), (c) are generalized causal models, the string diagram (a) is not since the output wire of $g$ is connected to two global outputs. Concerning \Cref{def:purebloom}, the generalized causal model (b) is not a pure bloom since the output of $\Lambda$ is not connected to a global output. The generalized causal model (c) is a pure bloom without inputs. For a further analysis of (c) regarding $d$-separation, see \Cref{ex:dSepExamples}\ref{ex:dSepExamplesI}. In all three examples, we have $\Sigma = G$ and assume $\mathsf{type}$ to be the identity map for simplicity.}
    \label{fig:ExCausalModels}
\end{figure}

For an example of a pure bloom morphism, we refer to \Cref{fig:ExCausalModels}. We will show that the soundness of the $d$-separation criterion holds for arbitrary generalized causal models (see \Cref{cor:dseparationCondIndependence}) while the completeness holds for pure bloom morphisms (see \Cref{thm:causalCompdSep}).

To define causal compatibility, we make the following assumption for the rest of the paper for convenience:

\begin{assumption}
	\label{strict}
	Throughout, $\cat{C}$ is a strict Markov category.
\end{assumption}

A monoidal category is strict if monoidal products like $A \otimes (B \otimes C)$ and $(A \otimes B) \otimes C$, or $B \otimes A$ and $A \otimes B$ are not only isomorphic but identical. We refer to \cite[Section~1.3]{He19} for an introduction to strict monoidal categories.

Although most examples like $\cat{FinStoch}$, $\cat{BorelStoch}$ or $\cat{Stoch}$ already fail strictness, this does not exclude these examples since we can always work with a strictification instead \cite[Theorem~10.17]{Fr20}, which satisfies \Cref{strict}.
On the other hand, our free Markov categories $\cat{FreeMarkov}_\Sigma$ already satisfy this condition ``on the nose''.
In any case, \Cref{strict} is a useful convenience that holds without loss of generality.

\begin{notation}
\label{not:CausalModel}
For the rest of the paper, we will assume that $\varphi$ is a generalized causal model with
\begin{equation*}
\varphi \quad \coloneqq \quad \begin{tikzcd} 
& G & \\
\underline{n} \arrow{ru}{p}  & & \underline{m} \arrow[swap]{lu}{q}
\end{tikzcd}
\end{equation*}
which becomes a cospan in $\cat{FinHyp} / \Sigma$ through $\mathsf{type} : G \to \Sigma$. 

We identify inputs and outputs with the wires they map to under $p$ and $q$ and refer to them as such. In particular, we define
\begin{align}
	\label{phi_defn}
	\In(\varphi) &\coloneqq p(\underline{n}) \subseteq W(G)\\
	\label{phi_defn2}
	\Out(\varphi) &\coloneqq q(\underline{m}) \subseteq W(G)
\end{align} 
for the set of all input/output wires. If $\varphi$ is a pure bloom morphism, then $\Out(\varphi) = W(G)$.
\end{notation}

\noindent
In the following, we present the notion of causal compatibility for a generalized causal model $\varphi$. Intuitively, a morphism $f$ in any Markov category $\cat{C}$ is compatible with $\varphi$ if we can plug in a morphism from $\cat{C}$ into every box in $B(\Sigma)$ in such a way that the composite is exactly $f$, and such that the global input and output wires of $\varphi$ correspond to a given tensor factorization of the domain and codomain of $f$:

\begin{definition}[Compatibility]
	\label{def:compatibility}
	For $\Sigma$ a hypergraph, let $\varphi$ be a generalized causal model as in \eqref{phi_morphism}.
	Let further
	\[
		f : \bigotimes_{i=1}^n W'_i \to \bigotimes_{j=1}^m V'_j
	\]
	be a morphism in any Markov category $\cat{C}$ satisfying \Cref{strict}, equipped with a fixed tensor decomposition of its domain and codomain as indicated.

	We call $f$ \emph{compatible with} $\varphi$ if there exists a strict Markov functor\footnote{i.e.\ a strict symmetric monoidal functor which preserves the comonoid structure.} $F: \cat{FreeMarkov}_{\Sigma} \to \cat{C}$ such that:
\begin{enumerate}[label=(\roman*)]
	\item We have
		\begin{equation}
			\label{F_wires_map}
			W'_i = F(\mathsf{type}(p(i))), \qquad V'_j = F(\mathsf{type}(q(j)))
		\end{equation}
		for all input indices $i = 1,\ldots,n$ and output indices $j = 1,\ldots,k$.
	\item $f = F(\varphi)$.
\end{enumerate}
\end{definition}
This generalizes the functorial definition of causal models as first studied in~\cite{Fo13}.

Note that the functor $F$ must assign to every wire type (i.e., wire in $\Sigma$) a corresponding object in the category $\cat{C}$. This implies that wires in $W(G)$ with identical types must map to the same object in $\cat{C}$. For example, one may consider a situation where $f$ is a probability distribution with no inputs and all output variables are real-valued. In this case, we have $V'_j = \mathbb{R}$ for all $j$, and one may want to consider a causal model $\varphi$ in which all wires are likewise of the same type. For the string diagrammatic picture of this definition, we refer to \Cref{ssec:CausalCompatibility}.

Similarly, the hypergraph morphism $\mathsf{type} : G \to \Sigma$ assigns to each box in $G$ a specific ``type'' box in $\Sigma$. This means that under $F$, any two boxes with the same type must map to the same morphism in $\cat{C}$.
This is why generalized causal models in our sense can naturally incorporate the condition that several causal mechanisms must be the same, namely when choosing the types in a way enforcing this. This has already been highlighted in our discussion of gs-monoidal string diagrams (see \Cref{ex:ConstructionGsMonoidal}).

In the following, we denote for every wire $X \in W(G)$ in $\varphi$ the corresponding object $F(\mathsf{type}(X))$ in $\cat{C}$ by $X'$.
Similarly, for every set of wires $\mathcal{W} \subseteq W(G)$ in $\varphi$, we denote the corresponding multiset of wires in $\cat{C}$ by $\mathcal{W}'$. For the rest of the paper, we will associate this multiset with the corresponding tensor product in $\cat{C}$ obtained by tensoring its elements, where we ignore the question of how to order the factors.

\section{Conditionals and conditional independence}
\label{sec:Conditional}

In this section, we review a central assumption to prove the $d$-separation criterion, namely the existence of conditionals. Moreover, we introduce a string diagrammatic definition of conditional independence. 

\subsection{Existence of conditionals}

To use the $d$-separation criterion to detect causal compatibility, we need in addition the existence of conditionals. This notion has been studied in categorical terms in \cite{Ch19b} in a special case, where the authors call it \emph{admitting disintegration}, and subsequently in \cite{Fr20} in general. In the following, we briefly review the definition following \cite[Section 11]{Fr20}.

\begin{definition}\label{def:existenceConditionals}
Let $\cat{C}$ be a Markov category. We say that $\cat{C}$ has \emph{conditionals} if for every morphism $f : \mathcal{A} \to \mathcal{X} \otimes \mathcal{Y}$, there is a morphism $f_{|_\mathcal{X}}$ such that
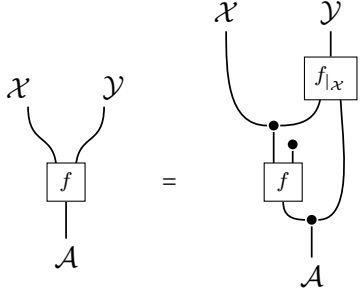
\begin{equation*}
\input{figures/ConditionalsDefinition1.tikz} \quad = \quad \input{figures/ConditionalsDefinition2.tikz}
\end{equation*} 
\end{definition}

Note that $\mathcal{X}$, $\mathcal{Y}$ are single output spaces. However, since these can arise as products of multiple output spaces each, we use notation that suggests interpreting $\mathcal{X}$ and $\mathcal{Y}$ as multisets of spaces. Since Markov categories are symmetric monoidal, we can use the multisets or tensor product pictures interchangeably.

Examples of categories having conditionals are $\cat{FinStoch}$, $\cat{Gauss}$ as well as $\cat{BorelStoch}$. In contrast, $\cat{Stoch}$ does not have conditionals (see \cite[Examples 11.6--11.8]{Fr20} and references therein).

In the following, we prove that the deterministic Markov category $\cat{FinSet}$ and the possibilistic Markov category $\cat{FinSetMulti}$ have conditionals.
\begin{proposition}\label{prop:SetMultiConditionals}
$\cat{FinSet}$ and $\cat{FinSetMulti}$ have conditionals.
\end{proposition}
\begin{proof} We show the statements separately, doing the case of single objects $A$, $X$ and $Y$ without loss of generality.
\begin{enumerate}[label=(\roman*)]
	\item In $\cat{FinSet}$, any map $f: A \to X \times Y$ is uniquely determined by its components $f_X : A \to X$ and $f_Y : A \to Y$, which are equivalently its marginals. Defining $f_{|_X}: X \times A \to Y$ as $f_{|_X}(x, a) \coloneqq f_X(a)$ trivially shows the statement.\footnote{The same argument shows that every cartesian monoidal category has conditionals.}
	\item Let $f : A \to X \times Y$ be a morphism in $\cat{FinSetMulti}$ with $f(x,y \, | \, a) \in \{0,1\}$ for $x \in X, y \in Y, a \in A$. Marginalizing over $Y$, we get
	\[
		f_Y(x \, | \, a) = \sum_{y \in Y} f(x,y \, | \, a) = \left\{\begin{array}{cl} 1 & \text{if } \exists y \in Y \text{ s.t. } f(x,y \, | \, a) = 1, \\ 0 & \text{else.} \end{array}\right.
	\]
where we have used the equation $1 + 1 = 1$.
Defining  $f(y \, | \, a,x) \coloneqq f(x,y \, | \, a)$, we have
$$f(x,y \, | \, a) = f(y \, | \, a,x) \cdot f(x \, | \, a),$$
which shows the statement.
\qedhere
\end{enumerate}
\end{proof}
Note that the proof also generalizes trivially to $\cat{Set}$ and $\cat{SetMulti}$.

\subsection{Conditional independence}

A second ingredient of the $d$-separation criterion is the definition of conditional independence. The following definition has been introduced in several works (see for example \cite{Ch19b,Co11}) and shown to satisfy the well-known semigraphoid properties. In addition, Fritz \cite{Fr20} shows that it is still meaningful to define conditional independence in the absence of conditionals.

\begin{definition}
\label{def:CondIndependence}
A morphism $f : I \to \mathcal{X} \otimes \mathcal{Z} \otimes \mathcal{Y}$ in $\cat{C}$ displays the \emph{conditional independence} $\mathcal{X} \perp \mathcal{Y} \ | \ \mathcal{Z}$ if it can be written as
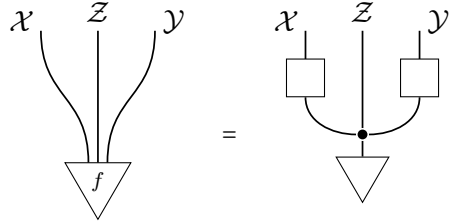
\begin{equation*}
\input{figures/tripartiteState.tikz} \quad = \quad \input{ConditionalIndependence.tikz}
\end{equation*} 
\end{definition}

In other words, $\mathcal{X} \perp \mathcal{Y} \ | \ \mathcal{Z}$ holds if $f$ is compatible with the generalized causal model which corresponds to the string diagram on the right-hand side (where all boxes are of distinct type, and we leave labels off for simplicity).

\begin{remark}
\label{rem:condIndepEmptyset}
The conditional independence $\mathcal{X} \perp \emptyset \ | \ \mathcal{Z}$ is equivalent to the existence of the conditional $r_{|_\mathcal{Z}}$. Therefore, if $\cat{C}$ has conditionals, then every state $r$ satisfies $\mathcal{X} \perp \emptyset \ | \ \mathcal{Z}$ for every tensor factorization of its codomain $\mathcal{X} \otimes \mathcal{Z}$. 
\end{remark}

While conditional independence coincides with the standard definition in probabilistic Markov categories like $\cat{FinStoch}$, $\cat{Gauss}$ or $\cat{BorelStoch}$, we show in the following an example of conditional independence for $\cat{FinSetMulti}$.

\begin{example}
\begin{enumerate}[label=(\roman*)]
	\item Let $X = Y = Z$ be finite sets and $f(x,y,z) = \delta_{x,y,z}$, where $\delta_{x,y,z} = 1$ if and only if $x = y = z$.
		In other words, a triple of outcomes $(x,y,z)$ is possible if and only if $x=y=z$. Then $f$ displays the conditional independence $X \perp Y \ | \ Z$ since
	$$f(x,y,z) = f(x \, | \, z) \cdot f(y \, | \, z) \cdot f(z),$$
	with $f(z) = 1$ for every $z \in Z$ and $f(x \, | \, z) = \delta_{x,z}$ and $f(y \, | \, z) = \delta_{y,z}$. Note that in this representation $f(x \, | \, z)$ and $f(y \, | \, z)$ correspond to the left and the right boxes in the decomposition of \Cref{def:CondIndependence}.
	
	\item Let $X = Y = Z$ be finite sets and similarly $f(x,y,z) = \delta_{x,y} \cdot \delta_{z,z_{0}}$ for fixed $z_0 \in Z$, which has marginal $f_Z(z) = \delta_{z,z_{0}}$. Then $f$ does not display the conditional independence $X \perp Y \ | \ Z$. Assuming the contrary, we would have $$\delta_{x,y} = f(x \, | \, z)  \cdot f(y \, | \, z)$$
	which cannot be the case as soon as our sets have at least two elements.
	
\end{enumerate}

\end{example}

As we have already done in the previous definition, we write $\mathcal{X}, \mathcal{Y}, \ldots$ for arbitrary lists of objects in $\cat{C}$.
We also allow for implicit reordering of these lists---effectively identifying these lists with multisets---and omit mention of the relevant compositions of $f$ by swap morphisms.\footnote{Of course, these kinds of bookkeeping mechanisms are also present in the traditional notation of probability distributions and measures, though rarely made explicit in that context.}
This allows us to talk about conditional independence with respect to any tripartition of the tensor factors in the codomain of any state $f$.

With this in mind, we now introduce a notion of conditional independence for morphisms with inputs. This notion is the key ingredient of the $d$-separation criterion presented in \Cref{ssec:causalCompatibility}, and it is the categorical generalization of the \emph{transitional conditional independence} introduced recently in \cite[Definition~3.1]{Fo21}.

\begin{definition}
\label{def:condInd_diff}
A morphism $f: \mathcal{A} \to \mathcal{X} \otimes \mathcal{Y} \otimes \mathcal{Z}$ in $\cat{C}$ displays the \emph{conditional independence} $\mathcal{X} \perp \mathcal{Y} \ | \ \mathcal{Z}$ if there exists a factorization of the form
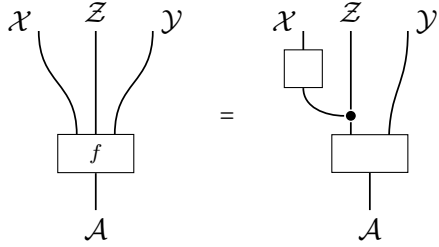
\begin{equation*}
\input{figures/tripartiteMorphism.tikz} \quad = \quad \input{figures/ConditionalIndependenceMorphism.tikz}
\end{equation*} 
\end{definition}

\begin{remark}
\label{rem:condInd}
Note that the above definition of conditional independence is not symmetric, i.e.\ $\mathcal{X} \perp \mathcal{Y} \ | \ \mathcal{Z}$ does not necessarily imply $\mathcal{Y} \perp \mathcal{X} \ | \ \mathcal{Z}$.
If $\cat{C}$ has conditionals then $\mathcal{X} \perp \mathcal{Y} \ | \ \mathcal{Z}$ rewrites to
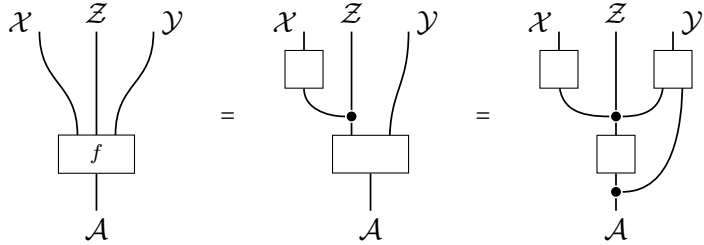
\begin{equation*}
\input{figures/tripartiteMorphism.tikz} \quad = \quad \input{figures/ConditionalIndependenceMorphism.tikz} \quad = \quad \input{figures/ConditionalIndependenceMorphismLeg.tikz}
\end{equation*}
which highlights the asymmetry. Moreover, if $\mathcal{A}$ is trivial, then the conditional independence coincides with \Cref{def:CondIndependence}.

Due to the asymmetry, the outputs in $\mathcal{Y}$ might contain information about the inputs $\mathcal{A}$ which cannot be retrieved just from the outputs in $\mathcal{Z}$. On the other hand, the outputs in $\mathcal{X}$ are generated using only the information from the outputs in $\mathcal{Z}$. The local Markov property that we will use in \Cref{def:MarkovProperties} explicitly highlights this asymmetry: the output of a box (corresponding to $\mathcal{X}$) is independent of its non-descendants ($\mathcal{Y}$) given its inputs ($\mathcal{Z}$). Every global input is non-descendant of any box; however, not every global input wire is directly an input to the box itself.  
\end{remark}

\begin{example}
\label{ex:CondIndSet}
	In the Markov category $\cat{Set}$, a function $f: A \to X \times Z \times Y$ is always such that each one of its components $f_X$, $f_Y$ and $f_Z$ is a deterministic function of $a \in A$. Such $f$ displays the conditional independence $X \perp Y \ | \ Z$ if and only if the $X$ output can be rewritten as a function of the $Z$ output only, i.e.\ $x = g(z)$, or equivalently if and only if $f_X$ factors through $f_Z$.

\end{example}

\section{Deciding causal compatibility with \texorpdfstring{$d$}{d}-separation}
\label{sec:decidingComp}

The main goal of this section is to prove that the $d$-separation criterion \cite[Section 1.2.3]{Pe09} correctly detects causal compatibility not just in discrete probability but in all Markov categories with conditionals. 
To this end, we introduce a categorical notion of $d$-separation phrased in terms of connectedness of the gs-monoidal string diagram representing the causal model. We then show that this notion coincides with the classical notion of $d$-separation whenever the latter applies.

This part is structured as follows. In \Cref{ssec:dseparation}, we introduce the categorical notion of $d$-separation on generalized causal models. Moreover, we show in \Cref{pro:dSepCoincide} that this notion coincides with the classical notion of $d$-separation for all those generalized causal models that correspond to DAGs. In \Cref{ssec:causalCompatibility}, we first show that $d$-separation implies conditional independences for compatible morphisms in Markov categories with conditionals. We then show in \Cref{thm:causalCompdSep} that $d$-separation fully characterizes causal compatibility.

\subsection{Categorical \texorpdfstring{$d$}{d}-separation}
\label{ssec:dseparation}

For a gs-monoidal string diagram
\begin{equation*}
	\varphi \,= \begin{tikzcd}
	& G & \\
	\underline{n} \arrow{ru}{p}  & & \underline{m} \arrow[swap]{lu}{q}
\end{tikzcd} 
\end{equation*} 
and a set of output wires $\mathcal{Z} \subseteq \Out(\varphi)$, we define a new gs-monoidal string diagram $\Cut_{\mathcal{Z}}(\varphi)$ obtained by removing the wires in $\mathcal{Z}$ in the following sense. Its underlying hypergraph $H$ is such that the set of boxes is the same, $B(H) \coloneqq B(G)$, while the set of wires is $W(H) \coloneqq W(G) \setminus \mathcal{Z}$.
Each box has the same input and output wires as before, expect in that those in $\mathcal{Z}$ are simply removed, which lowers the arities of the boxes correspondingly.
We also remove all occurrences of wires in $\mathcal{Z}$ from the global inputs and outputs, and this results in a gs-monoidal string diagram
\begin{equation*}
	\Cut_{\mathcal{Z}}(\varphi) \coloneqq \begin{tikzcd}
	& H & \\
	\underline{n}' \arrow{ru}{p'}  & & \underline{m}' \arrow[swap]{lu}{q'}
\end{tikzcd}
\end{equation*}
where $\underline{n}' \subseteq \underline{n}$, $\underline{m}' \subseteq \underline{m}$ and $p'$, $q'$ are the restrictions of $p$ and $q$, respectively.
Note that $\Cut_{\mathcal{Z}}(\varphi)$ is generally not a morphism in $\cat{FreeMarkov}_{\Sigma}$ anymore since it is not normalized. However, it can be understood as a morphism in $\cat{FreeGS}_{H}$. \Cref{ex:categoricalSeparation} will present a few examples.

We next introduce some notation for paths of wires.

\begin{definition}
\label{defn:categoricalSep}
Let $\varphi$ be a gs-monoidal string diagram in $\cat{FreeGS}_{\Sigma}$.
\begin{enumerate}[label=(\roman*)]
	\item\label{path_def} An \emph{undirected path} between two wires $X,Y \in W(G)$ is a sequence of wires
		\[
			X = W_1, \, W_2, \, \ldots, \, W_{n}, \, W_{n+1} = Y
		\]
		together with a sequence of boxes $b_1, \ldots, b_n \in B(G)$ such that
		\begin{equation*}
		\In(b_i, W_i) + \Out(b_i, W_i) \geq 1 \quad \text{ and } \quad  \In(b_i, W_{i+1}) + \Out(b_i, W_{i+1}) \geq 1.
		\end{equation*} 
	If there exists an undirected path between $X$ and $Y$, then we write $X \minus Y$.
	\item\label{underlying_DAG} For two wires $A, B \in W(G)$, we write $A \rightarrow B$ if there exists a box $b \in B(G)$ such that
		\begin{align} \In(b, A) = 1 \quad \text{ and } \quad \Out(b, B) = 1.\end{align} 
	\item\label{doubleHead} For two wires $A,B \in W(G)$, we write $A \twoheadrightarrow B$ if there exists a sequence of wires $W_1, \ldots, W_{n} \in W(G)$ such that
		\begin{equation}
		\label{eq:directedPath} A \rightarrow W_1 \rightarrow \ \ldots \ \rightarrow W_n \rightarrow B.
		\end{equation} 
\end{enumerate}
\end{definition}

Thus, an undirected path in $\varphi$ may traverse a box not just from input or output or vice versa, but also from input to input or output to output.

The intuitive idea behind the following definition of $d$-separation, as already briefly discussed in~\cite[Remark~7.2]{Fr22}, was communicated to us by Rob Spekkens.

\begin{definition}[Categorical $d$-separation]
	Let $\varphi$ be a generalized causal model. For three disjoint sets of output wires $\mathcal{X}, \mathcal{Y}, \mathcal{Z} \subseteq \Out(\varphi)$, we say that $\mathcal{Z}$ \emph{$d$-separates} $\mathcal{X}$ and $\mathcal{Y}$ if
\begin{equation*}
\Cut_{\mathcal{Z}}(\varphi_{\mathcal{X}, \mathcal{Y}, \mathcal{Z}})
\end{equation*} 
has no undirected path between any output in $\mathcal{X}$ and any output in $\mathcal{Y}$.
\end{definition}

Here, $\varphi_{\mathcal{W}} \coloneqq \norm(\Del{\mathcal{W}^{c}} \circ \varphi)$ with $\mathcal{W}^{c} = \Out(\varphi) \setminus \mathcal{W}$ denotes the marginal on $\mathcal{W}$ in $\cat{FreeMarkov}_{\Sigma}$.
The absence of an undirected path as in the definition manifests itself in the string diagrams simply as topological disconnectedness.

\begin{example}
	The following examples constitute the basic components of ``classical'' $d$-separation\footnote{See~\cite{JaZa17} for an account of chains, forks and colliders featuring both numerical examples and a categorical formalism.} and illustrate the simplicity of categorical $d$-separation.
In all cases, the unlabeled boxes denote distinct generators, i.e.~distinct boxes in the generating hypergraph $\Sigma$.
\label{ex:categoricalSeparation}
\begin{enumerate}[label=(\roman*)]
	\item Fork: consider the morphism
	\begin{equation*}
	\varphi \quad = \quad \input{figures/tripartiteDSeparationExample.tikz}
	\end{equation*} 
	Then $Z$ $d$-separates $X$ from $Y$ since
	\begin{equation*} 
	\Cut_{Z}(\varphi) \quad = \quad \input{figures/tripartiteDSeparationExample2.tikz}
	\end{equation*} 
	has disconnected $X$ and $Y$.
	\item Chain: consider the morphism
	\begin{equation*}
	\varphi \quad = \quad \input{figures/sequentialDSeparationExample.tikz}
	\end{equation*} 
	Then $Z$ $d$-separates $X$ from $Y$ since
	\begin{equation*} 
	\Cut_{Z}(\varphi) \quad = \quad \input{figures/sequentialDSeparationExample2.tikz}
	\end{equation*} 
	has disconnected $X$ and $Y$.
	\item Collider: consider the morphism
	\begin{equation*}
	\varphi \quad = \quad \input{figures/NonConectednessExample1.tikz} 
	\end{equation*} 
	Then $Z$ does \emph{not} $d$-separate $X$ from $Y$ since we have
	\begin{equation*}
	\varphi_{X,Z,Y} \quad = \quad \input{figures/NonConectednessExample4.tikz} \quad =  \quad \input{figures/NonConectednessExample5.tikz}
	\end{equation*} 
	and therefore
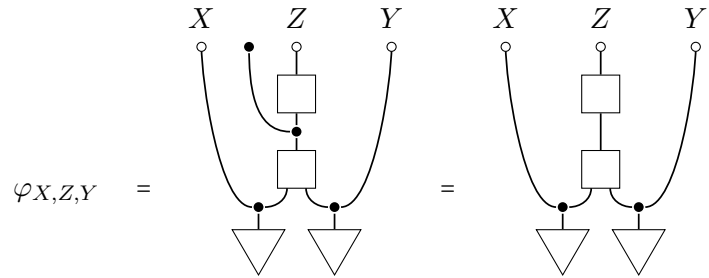
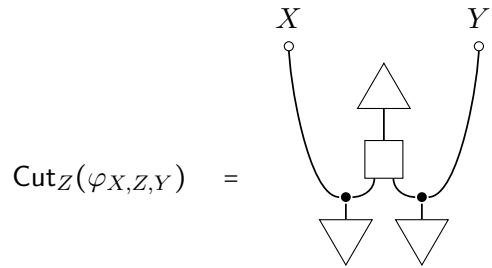
	\begin{equation*}
	\Cut_{Z}(\varphi_{X,Z,Y}) \quad = \quad \input{figures/NonConectednessExample6.tikz}
	\end{equation*} 
	which still contains an undirected path $X - Y$. The same reasoning applies when $\mathcal{Z} = \{W\}$ or $\mathcal{Z} = \{W, Z\}$. However, if $\mathcal{Z} = \emptyset$, then
	\begin{equation*}
	\varphi_{X,Y} \quad = \quad \input{figures/NonConectednessExample2_unlabelled.tikz} \quad = \quad \input{figures/NonConectednessExample3_unlabelled.tikz}
	\end{equation*} 
	which has disconnected $X$ and $Y$.
	Therefore $\emptyset$ $d$-separates $X$ and $Y$.
	\item Consider the morphism
	\begin{equation*}
	\varphi \quad = \quad \input{figures/tripartiteDSeparationExampleMarginal.tikz}
	\end{equation*} 
	This morphism differs from (i) by adding an independent process $A \to W$. The normalized marginal $\varphi_{X,Y,Z}$ is given by
	\begin{equation*}
	\varphi_{X,Y,Z} \quad = \quad \input{figures/tripartiteDSeparationExampleMarginal2.tikz}
	\end{equation*} 
	which again shows that $Z$ $d$-separates $X$ from $Y$ since cutting $Z$ makes $X$ and $Y$ disconnected.

\end{enumerate}
\end{example}

In order to define the classical notion of $d$-separation, we note that certain gs-monoidal string diagrams have an underlying DAG, given by using wires as nodes and taking the edges to be $\to$ as in~\Cref{defn:categoricalSep}.
We use the term \emph{DAG path} to refer to an undirected path in this DAG, i.e.~to a sequence of wires connected by boxes from input to output or vice versa (but \emph{not} from input to input or output to output). For example, in the string diagram
$$\input{figures/DAGpath.tikz}$$
there are two undirected paths witnessing $W_1 - W_3$, namely on the one hand the direct path $W_1, W_3$ and on the other hand the indirect $W_1, W_2, W_3$. Only the second is a DAG path since its two steps are input to output and output to input, while going from $W_1$ to $W_3$ directly would be input of $b$ to input of $b$.
We also define the ancestor wires of a given set of wires as
\begin{equation*}
\An(\mathcal{X}) = \{U \in W(G): \exists X \in \mathcal{X} \textrm{ such that } U \twoheadrightarrow X \}
\end{equation*} 
and the set of descendant wires as
\begin{equation*}
\Dec(\mathcal{X}) = \{U \in W(G): \exists X \in \mathcal{X} \textrm{ such that } X \twoheadrightarrow U \},
\end{equation*} 
where $A \twoheadrightarrow B$ is defined in~\Cref{defn:categoricalSep}\ref{doubleHead}.
Note that $\mathcal{X} \subseteq \An(\mathcal{X}), \Dec(\mathcal{X})$.
To state the following classical definition \cite[Definition 1.2.3]{Pe09} in our language, we restrict further to those generalized causal models that are determined by their underlying DAGs.
As noted in \Cref{rem:DAGRep}, a causal structure $\varphi$ is represented by a DAG, if the string diagram is pure bloom, every box has exactly one output, and it has no global inputs, i.e.\ $\In(\varphi) = \emptyset$.

\begin{definition}[Classical $d$-separation]
\label{def:classicalDseparation}
Let $\varphi$ be a pure bloom causal model with $\In(\varphi) = \emptyset$ and such that every box has exactly one output. Then:
\begin{enumerate}[label=(\alph*)]
	\item A DAG path $p$ in $\varphi$ is called \emph{$d$-separated} by a set of wires $\mathcal{Z} \subseteq \Out(\varphi)$ if at least one of
the following two conditions holds:
\begin{enumerate}[label=(\roman*)]
	\item $p$ contains a chain $W \rightarrow Z \rightarrow U$ or a fork $W \leftarrow Z \rightarrow U$ for some $Z \in \mathcal{Z}$.
	\item $p$ contains a collider $W \rightarrow M \leftarrow U$ where $M \notin \An(\mathcal{Z})$.
\end{enumerate}
	\item $\mathcal{X}$ is $d$-separated from $\mathcal{Y}$ by $\mathcal{Z}$ if every DAG path between every $X \in \mathcal{X}$ and $Y \in \mathcal{Y}$ is $d$-separated by $\mathcal{Z}$.
\end{enumerate}
\end{definition}

We will now prove the equivalence of categorical $d$-separation with classical $d$-separation for the class of causal models on which the latter is defined. This requires some preparation.

\begin{lemma}
\label{lem:pureBloomDiscard}
Let $\varphi$ be a pure bloom causal model, $b \in B(G)$ a box in $\varphi$ and $\mathcal{W} \subseteq \Out(\varphi)$ a subset of its wires. The following statements are equivalent:
\begin{enumerate}[label=(\roman*)]
	\item\label{bloomlem1} $\Out(b) \cap \An(\mathcal{W}) = \emptyset$.
	\item\label{bloomlem2} $b$ gets discarded in $\varphi_\mathcal{W} = \norm(\Del{\mathcal{W}^{c}} \circ \varphi)$.
\end{enumerate}
\end{lemma}
\begin{proof}
	$\ref{bloomlem2} \Longrightarrow \ref{bloomlem1}$: To prove the contrapositive, assume $\exists A \in \Out(b)$ such that $A \in \An(\mathcal{W})$. Then there is a path $A \twoheadrightarrow W$ with $W \in \mathcal{W}$. Since $W$ is still an overall output that does not get discarded, this path is still valid in $\Del{\mathcal{W}^{c}} \circ \varphi$. Therefore $b$ remains in $\norm(\Del{\mathcal{W}^{c}} \circ \varphi)$.

	$\ref{bloomlem1} \Longrightarrow \ref{bloomlem2}$: Consider the set $\Dec(\Out(b))$. By assumption, we have that $\Dec(\Out(b)) \cap \An(\mathcal{W}) = \emptyset$. We show that the box $b$ gets discarded in $\norm(\Del{\Dec(\Out(b))} \circ \varphi)$, which is enough because of $\mathcal{W}^c \supseteq \Dec(\Out(b))$.
By definition of $\Dec(\Out(b))$, there exists a final box\footnote{As defined in~\cite{Fr22}, a final box is one whose outputs are global outputs of $\varphi$ without further copy or discard. Such a box always exists since $\varphi$ is pure bloom and normalized (compare \cite[Lemma 4.6]{Fr22})} $\hat{b}$ such that $\Out(\hat{b}) \subseteq \Dec(\Out(b))$. This shows that $\hat{b}$ gets discarded in $\norm(\Del{\Dec(\Out(\hat{b}))} \circ \varphi)$.

	Define $\widetilde{\varphi} \coloneqq \norm(\Del{\Dec(\Out(\hat{b}))} \circ \varphi)$. Repeating the above procedure with $\widetilde{\varphi}$, we arrive after finite number of steps at $b$ itself being a final box. Since it is then eliminable after composing with $\Del{\Dec(\Out(b))}$, it indeed no longer appears in the normalization.
\end{proof}

We now show the promised equivalence result between categorical $d$-separation and classical $d$-separation in the cases where $\varphi$ represents a causal structure given by a DAG, i.e. $\In(\varphi) = \emptyset$ and every box has a single output.

\begin{proposition}
\label{pro:dSepCoincide}
	Both concepts of $d$-separation coincide on pure bloom causal models $\varphi$ with $\In(\varphi) = \emptyset$ and in which every box has exactly one output.
\end{proposition}

\begin{proof}
	To make the proof more intuitive, we introduce the term \emph{$d$-connected} as the negation of $d$-separated (in either version).

	We start by showing that classical $d$-connectedness implies categorical $d$-connectedness. Let $p$ be a DAG path between some $X \in \mathcal{X}$ and some $Y \in \mathcal{Y}$ which witnesses that $\mathcal{Z}$ makes $\mathcal{X}$ and $\mathcal{Y}$ be $d$-connected in the classical sense, which means that the following hold:
	\begin{enumerate}[label=(\roman*)]
		\item\label{connect_chain_fork} For every chain $W \rightarrow M \rightarrow U$ or fork $W \leftarrow M \rightarrow U$ in $p$, we have $M \not\in \mathcal{Z}$.
		\item\label{connect_collider} For every collider $W \rightarrow M \leftarrow U$ in $p$, we have $M \in \An(\mathcal{Z})$.
	\end{enumerate}
	For simplicity, we also assume without loss of generality that $p$ contains only one wire from $\mathcal{X}$ and $\mathcal{Y}$ each, say $X$ and $Y$ respectively.
	Then this $p$ can also be interpreted as an undirected path in $\varphi$, but generally not in $\varphi_{\textrm{cut}} \coloneqq \Cut_{\mathcal{Z}}(\varphi_{\mathcal{X}, \mathcal{Y}, \mathcal{Z}})$ since it may traverse wires that are not in $\varphi_{\textrm{cut}}$.
	However, we now show that there still is an undirected path $p'$ between $X$ and $Y$ in $\varphi_{\textrm{cut}}$.
	Since $p$ is $d$-connected, if $p$ contains a wire $Z \in \mathcal{Z}$, then it has to arise from a collider $U \rightarrow Z \leftarrow W$ in $p$.
	Removing wire $Z$ from $p$ still defines a valid \emph{undirected} path between $X$ and $Y$, pictorially:
	\begin{equation*} 
	\input{figures/dSepGadget1.tikz} \qquad \Longrightarrow \qquad \input{figures/dSepGadget1CUT.tikz}
	\end{equation*} 
	We prove that the path $p'$ obtained by removing all wires in $\mathcal{Z}$ from $p$ like this is an undirected path in $\varphi_{\textrm{cut}}$, which implies categorical $d$-connectedness. To this end, it only remains to show that each wire in $p'$ is an existing wire in $\varphi_{\textrm{cut}}$, which we do as follows:
\begin{enumerate}[label=(\roman*)]
	\item $X$ and $Y$ themselves are still in $\varphi_{\textrm{cut}}$.
	\item Every $Z \in \mathcal{Z}$ in $p$ is part of a collider $U \rightarrow Z \leftarrow W$ as above, so that $U, W \in \An(\mathcal{Z})$.
		This implies that $U$ and $W$ survive in $\varphi_{\mathcal{X},\mathcal{Y},\mathcal{Z}}$ by \Cref{lem:pureBloomDiscard}.
	\item Since $U$ and $W$ are themselves either the middle node in a chain or fork or the start or end of $p$, we have $U, W \not \in \mathcal{Z}$.
		This implies $U, W \in \An(\mathcal{Z}) \setminus \mathcal{Z}$, and therefore $U$ and $W$ survive also in $\varphi_{\textrm{cut}}$.
	\item For every chain $W \rightarrow M \rightarrow U$ in $p$, if $U$ survives in $\varphi_{\textrm{cut}}$, then so does $M$ (since it survives in $\varphi_{\mathcal{X},\mathcal{Y},\mathcal{Z}}$ and $M \not\in \mathcal{Z}$).
	\item For every fork $W \leftarrow M \rightarrow U$ in $p$, if $U$ or $W$ survives in $\varphi_{\textrm{cut}}$, then so does $M$ (since it survives in $\varphi_{\mathcal{X},\mathcal{Y},\mathcal{Z}}$ and $M \not\in \mathcal{Z}$).
\end{enumerate}
Since the wires in $p'$ are exactly those of $p$ minus some of the colliders, we can start with the first two observations and then apply the latter two repeatedly on any segment bounded by colliders or the starting node $X$ or the final node $Y$ in order to conclude that all wires in $p'$ are present in $\varphi_{\textrm{cut}}$.
This concludes one direction of the proof.

The converse direction of showing that categorical $d$-connectedness implies classical $d$-connectedness works similarly. Let $p$ be an undirected path between $X \in \mathcal{X}$ and $Y \in \mathcal{Y}$ in $\varphi_{\textrm{cut}}$. We assume without loss of generality that all wires in $p$ are distinct.
Furthermore, we also assume without loss of generality that $p$ is of the form, in terms of the notation from \Cref{defn:categoricalSep}\ref{doubleHead},
\begin{equation}
	\label{no_bad_colliders}
	X \twoheadleftarrow A - B \twoheadrightarrow Y, 
\end{equation}
where every wire that is in between $A$ and $B$ is not contained in $\An(\mathcal{X})$ or $\An(\mathcal{Y})$, or equivalently that every wire in $p$ that is also in $\An(\mathcal{X})$ is directly reached from $X$ by output-to-input traversals in $p$, and similarly for all wires in $\An(\mathcal{Y})$.
This property can be achieved by taking every wire in $p$ which is additionally in $\An(\mathcal{X})$ and replace it with the path from $X$ to it by a sequence of output-to-input traversals, and similarly for every wire in $\An(\mathcal{Y})$.
Note that this replacement may involve changing the starting and ending wires $X$ and $Y$ as well.

In order to turn $p$ into a DAG path $p'$ that witnesses classical $d$-separation, we need to remove all direct input-to-input traversals of a box in $p$; direct output-to-output traversals cannot occur due to the assumption that every box has exactly one output.
We can hence simply add to $p$ the unique output wire of every box that has an input-to-input traversal in $p$, and we obtain a valid DAG path $p'$.

It remains to verify the conditions on chains, forks and colliders. Clearly, $p'$ does not contain any chain $W \to Z \to U$ or fork $W \leftarrow Z \to U$ with $Z \in \mathcal{Z}$ since such a configuration cannot occur in $p$ to begin with. For a collider $W \to M \leftarrow U$, the unique box which outputs $M$ must be contained in $\varphi_{\mathcal{X},\mathcal{Y},\mathcal{Z}}$, and therefore be in $\An(\mathcal{X} \cup \mathcal{Y} \cup \mathcal{Z})$ by \Cref{lem:pureBloomDiscard}. However, $M$ being in $\An(\mathcal{X})$ or $\An(\mathcal{Y})$ violates the assumption that $p$ is of the form~\eqref{no_bad_colliders}. Therefore $M$ has to be in $\An(\mathcal{Z})$, showing the collider condition \ref{connect_collider}.
\end{proof}

We record one more observation on categorical $d$-separation for further use below.

\begin{lemma}
\label{lem:dSepInOutBox}
Let $\varphi$ be a pure bloom causal model and $\mathcal{X}, \mathcal{Y}, \mathcal{Z} \subseteq \Out(\varphi)$ a partition of all output wires such that $\mathcal{Z}$ categorically $d$-separates $\mathcal{X}$ and $\mathcal{Y}$. Then every box $b \in B(G)$ in $\varphi$ satisfies at least one of the cases:
\begin{enumerate}[label=(\roman*)]
	\item\label{sepauxX} $\In(b), \Out(b) \subseteq \mathcal{X} \cup \mathcal{Z}$.
	\item\label{sepauxY} $\In(b), \Out(b) \subseteq \mathcal{Y} \cup \mathcal{Z}$.
\end{enumerate}
\end{lemma}
\begin{proof}
	If there exist $Y \in \mathcal{Y} \cap \Out(b)$ and $X \in \mathcal{X} \cap \Out(b)$, then these wires are still in the output of $b$ in $\varphi_{\textrm{cut}}$, and this contradicts the assumed disconnectedness of $\varphi_{\textrm{cut}}$ with respect to $\mathcal{X}$ and $\mathcal{Y}$. Since $\mathcal{X}, \mathcal{Y}, \mathcal{Z}$ form a partition, this shows that either $\Out(b) \subseteq \mathcal{X} \cup \mathcal{Z}$ or $\Out(b) \subseteq \mathcal{Y} \cup \mathcal{Z}$.

Since $\varphi$ is a pure bloom we have that $\In(b) \subseteq \Out(\varphi)$ which guarantees that $\In(b) \subseteq \mathcal{X} \cup \mathcal{Y} \cup \mathcal{Z}$. Proving $\Out(b) \subseteq \mathcal{X} \cup \mathcal{Z} \Rightarrow \In(b) \subseteq \mathcal{X} \cup \mathcal{Z}$ and $\Out(b) \subseteq \mathcal{Y} \cup \mathcal{Z} \Rightarrow \In(b) \subseteq \mathcal{Y} \cup \mathcal{Z}$ works similarly to the first part of the proof, and this then proves the statement.
\end{proof}

Pictorially, \Cref{lem:dSepInOutBox} shows that if $\mathcal{Z}$ $d$-separates $\mathcal{X}$ and $\mathcal{Y}$, then every box $b$ in $\varphi$ is of the form
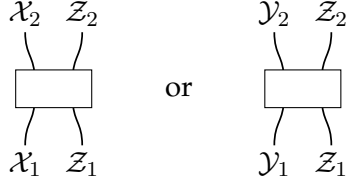
\begin{equation*}
\input{figures/boxXZ.tikz} \qquad \text{ or } \qquad \input{figures/boxYZ.tikz} 
\end{equation*} 
where $\mathcal{X}_i \subseteq \mathcal{X}$, $\mathcal{Y}_i \subseteq \mathcal{Y}$, $\mathcal{Z}_i \subseteq \mathcal{Z}$.

\subsection{Causal compatibility}
\label{ssec:causalCompatibility}

In the following, we show that $d$-separation implies conditional independence for any generalized causal model. We first prove this result for a partition of wires in a pure bloom causal model in \Cref{lem:dSepCondIndep}. We then refine it to any disjoint collection of wires in \Cref{cor:dseparationCondIndependence} in any generalized causal model. Finally, we show in \Cref{thm:causalCompdSep} that $d$-separation fully characterizes causal compatibility for pure bloom causal models in all Markov categories with conditionals.

Throughout, we also use the following convenient notation: If a morphism $f$ in $\cat{C}$ is compatible with a causal model $\varphi$ in the sense of \Cref{def:compatibility}, then we refer to the wires of $\varphi$ to indicate conditional independence instead of the objects in the tensor factorization of $f$. In other words, instead of writing $\mathcal{X}' \perp \mathcal{Y}' \ | \ \mathcal{Z}'$, we simply write $\mathcal{X} \perp \mathcal{Y} \ | \ \mathcal{Z}$. Here, each $W' = F(\mathsf{type}(W))$ is the object in $\cat{C}$ associated with the wire $W$ by the causal model functor $F$ (see \Cref{def:compatibility}).

\begin{lemma} \label{lem:dSepCondIndep}
Let $\cat{C}$ be a strict Markov category with conditionals, and let $\varphi$ be a pure bloom causal model. Further, let $\mathcal{X}, \mathcal{Y}, \mathcal{Z} \subseteq \Out(\varphi)$ be a partition of wires in $\varphi$ such that $\In(\varphi) \subseteq \mathcal{Y} \cup \mathcal{Z}$ and $\mathcal{X}$ and $\mathcal{Y}$ are $d$-separated by $\mathcal{Z}$.
	
If a morphism $f$ in $\cat{C}$ is compatible with $\varphi$, then $\mathcal{X} \perp \mathcal{Y} \ | \ \mathcal{Z}$ (as in \Cref{def:condInd_diff}).
\end{lemma}

\begin{proof}
	Choose a total ordering of all boxes $b_1, \ldots, b_{k-1} \in B(G)$ and a chain of sets of wires in $\Out(\varphi)$,
	\begin{equation*} 
	\In(\varphi) = \mathcal{W}_1 \subseteq \ldots \subseteq \mathcal{W}_k = \Out(\varphi),
	\end{equation*}
	such that $\mathcal{W}_{i+1} = \Out(b_i) \cup \mathcal{W}_i$ and $\An(\mathcal{W}_i) = \mathcal{W}_i$. Note that there is a factorization
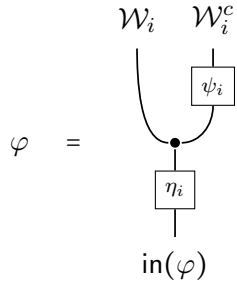
\begin{equation*}
\varphi \quad = \quad \input{figures/morphismFactorization.tikz}
\end{equation*} 
in $\cat{FreeMarkov}_\Sigma$, where $\eta_i$ is again a pure bloom and $\eta_1$ is an identity morphism.
	The existence of such a chain of sets follows easily by induction on the number of boxes based on the existence of a final box \cite[Lemma 4.6]{Fr22}.\footnote{It is also worth noting that for causal models which correspond to DAGs, this statement amounts to the standard fact that every DAG can be refined to a total order.}

	Then for every $i \in \{1, \ldots, k\}$, we show the existence of a decomposition\footnote{Note that we assume, without loss of generality, that every box in the string diagram depends on every variable in the collection of wires. For example, $k_i$ depends on every variable in $\mathcal{Z} \cap \mathcal{W}_i$. If this is not the case, say, $k_i$ does not depend on $W \in \mathcal{Z} \cap \mathcal{W}_i$, we replace it by $k_i' = k_i \circ (\mathrm{id} \otimes \Del{W})$.}
\begin{equation}
	\label{eq:goalDecompose}
	f \quad = \quad \input{figures/dSepCriterionProof1.tikz}
\end{equation} 
Since $\mathcal{W}_k^{c} = \emptyset$, setting $i = k$ proves the desired statement.

	We prove this stronger claim by induction on $i$. The start of the induction at $i = 1$ is trivial since $\eta_1$ is the identity and therefore
	\begin{equation*}	
	f \quad = \quad \input{figures/dSepCriterionProof1a.tikz}
	\end{equation*}
	since $\mathcal{X} \cap \mathcal{W}_1 = \mathcal{X} \cap \In(\varphi) = \emptyset$. 	For the induction step, we prove the statement at $i+1$. Since $\varphi$ is pure bloom and since $\An(\mathcal{W}_{i+1}) = \mathcal{W}_{i+1}$, we can peel off the box $b_i$ with $\mathcal{W}_{i+1} = \Out(b_i) \cup \mathcal{W}_i$ from $\psi_i$, so as to achieve the decomposition
\begin{equation} \label{eq:dSepInductionStep}
f \quad = \quad \input{figures/dSepCriterionProof2.tikz}
\end{equation} 
where we have used the induction assumption to obtain a decomposition as in the lower half, and the dashed wires indicate that only some of them may be present since the inputs of $b_i$ are an unspecified subset of $\mathcal{W}_i$.

	By \Cref{lem:dSepInOutBox} we have to distinguish two cases:
\begin{enumerate}[label=(\roman*)]
	\item $\In(b_i),\Out(b_i) \subseteq \mathcal{X} \cup \mathcal{Z}$. Then, the third dashed wire in the above decomposition of $f$ is not needed, and we consider the morphism
	\begin{equation*}
	g \quad \coloneqq \quad \input{figures/dSepCriterionProof3.tikz}
	\end{equation*} 
	which is part of that decomposition.
	By the existence of conditionals, we can rewrite $g$ in the form
	\begin{equation*}
	g \quad = \quad \input{figures/dSepCriterionProof4.tikz}	\quad = \quad \input{figures/dSepCriterionProof5.tikz}
	\end{equation*} 
	where both lower boxes can be refined with internal structure consisting of carrying $(\mathcal{Z} \cap \mathcal{W}_i)'$ forward on a separate wire, but this internal structure is not relevant for the remainder of the proof.
	Substituting this form of $g$ into Equation \eqref{eq:dSepInductionStep}, i.e.\ replacing the morphism $k_i$ there with the box $r_i$ here and merging the box $\ell_i$ here with $h_i$ there, proves the induction step.
	\item $\In(b_i),\Out(b_i) \subseteq \mathcal{Y} \cup \mathcal{Z}$. Then, the first dashed wire in the above decomposition of $f$ is not needed, and we can merge $F(b_i)$ with $h_i$, which shows the statement.
\end{enumerate}
\end{proof}

We will now generalize \Cref{lem:dSepCondIndep} to all generalized causal models and to arbitrary disjoint sets $\mathcal{X}$, $\mathcal{Y}$ and $\mathcal{Z}$ which do not necessarily partition the set of all wires.

\begin{lemma}
\label{lem:extensiontripartition}
Let $\varphi$ be a generalized causal model and $\mathcal{X}$, $\mathcal{Y}$, $\mathcal{Z} \subseteq \Out(\varphi)$ a tripartition of output wires in $\varphi$ such that $\In(\varphi) \subseteq \mathcal{Y} \cup \mathcal{Z}$ and such that $\mathcal{Z}$ categorically $d$-separates $\mathcal{X}$ and $\mathcal{Y}$. Then there exists a tripartition of wires $\widetilde{\mathcal{X}} \supseteq \mathcal{X}$, $\widetilde{\mathcal{Y}} \supseteq \mathcal{Y}$, $\mathcal{Z}$ in the pure bloom version $\varphi_{\textrm{pure-bloom}}$ of $\varphi$\footnote{The pure-bloom version $\varphi_{\textrm{pure-bloom}}$ is obtained by copying each wire so to make it into an output. It is part of the bloom-circuitry factorization of \cite{Fr22}.} such that
\begin{equation*}
\mathcal{Z} \text{ $d$-separates } \widetilde{\mathcal{X}} \text{ and } \widetilde{\mathcal{Y}} \qquad \text{ in } \varphi_{\textrm{pure-bloom}}
\end{equation*}
\end{lemma}

\begin{proof}
With $\varphi_{\textrm{cut}} \coloneqq \Cut_{\mathcal{Z}}(\varphi_{\textrm{pure-bloom}})$, define
\begin{equation*}
\widetilde{\mathcal{X}} \coloneqq \{U \in \Out(\varphi_{\textrm{cut}}): \exists X \in \mathcal{X}: X - U \textrm{ in } \varphi_{\textrm{cut}} \} \supseteq \mathcal{X}
\end{equation*} 
to be the connected component of $\mathcal{X}$ in $\varphi_{\textrm{cut}}$, and
\begin{equation*}
\widetilde{\mathcal{Y}} \coloneqq \Out(\varphi_{\textrm{pure-bloom}}) \setminus \left(\widetilde{\mathcal{X}} \cup \mathcal{Z}\right) \supseteq \mathcal{Y}. 
\end{equation*}
By definition, $\widetilde{\mathcal{X}}, \widetilde{\mathcal{Y}}, \mathcal{Z}$ form a tripartition of wires in $\varphi_{\textrm{pure-bloom}}$. Moreover,
$\widetilde{\mathcal{X}}$ and $\widetilde{\mathcal{Y}}$ are categorically $d$-separated by $\mathcal{Z}$ since any path in $\Cut_{\mathcal{Z}}(\varphi_{\textrm{pure-bloom}})$ is a valid path in $\Cut_{\mathcal{Z}}(\varphi)$ and vice versa.
\end{proof}

\begin{corollary}
\label{cor:dseparationCondIndependence}
	Let $\cat{C}$ be a strict Markov category with conditionals, and let $\varphi$ be a generalized causal model.\footnote{In this situation $\varphi$ does not need to be pure bloom.} Further, let $\mathcal{X}, \mathcal{Y}, \mathcal{Z} \subseteq \Out(\varphi)$ be disjoint sets of output wires in $\varphi$ such that $\In(\varphi) \subseteq \mathcal{Y} \cup \mathcal{Z}$ and $\mathcal{X}$ and $\mathcal{Y}$ are $d$-separated by $\mathcal{Z}$. If $f$ is compatible with $\varphi$, then $\mathcal{X} \perp \mathcal{Y} \ | \ \mathcal{Z}$.
\end{corollary}

In this statement, we use another standard convention: when the disjoint sets $\mathcal{X},\mathcal{Y},\mathcal{Z}$ do not partition the set of wires of $\varphi$, then the conditional independence $\mathcal{X} \perp \mathcal{Y} \ | \ \mathcal{Z}$ is to be understood as \Cref{def:condInd_diff} applied to the corresponding marginal 
$$f_{\mathcal{X}',\mathcal{Y}',\mathcal{Z}'} = \Del{(\mathcal{X}' \cup \mathcal{Y}' \cup \mathcal{Z}')^{c}} \circ f$$
rather than to $f$ itself.

\begin{proof}
We prove this statement by reducing it to the case of pure bloom causal models treated in \Cref{lem:dSepCondIndep}.

Consider the restricted causal model $\psi \coloneqq \varphi_{\mathcal{X}, \mathcal{Y}, \mathcal{Z}}$. We have that $f_{\mathcal{X}',\mathcal{Y}',\mathcal{Z}'} = F(\psi)$ since $F$ is a Markov functor. By the definition of categorical $d$-separation, $\mathcal{Z}$ $d$-separates $\mathcal{X}$ and $\mathcal{Y}$ also in $\psi$. Let $\psi_{\textrm{pure-bloom}}$ be the pure bloom version of $\psi$. Since $f_{\mathcal{X}',\mathcal{Y}',\mathcal{Z}'}$ is compatible with $\psi$, we can extend it to a pure bloom version
\begin{equation*}
g \coloneqq F(\psi_{\textrm{pure-bloom}})
\end{equation*}
of which $f_{\mathcal{X}',\mathcal{Y}',\mathcal{Z}'}$ is a marginal.

By \Cref{lem:extensiontripartition}, for $\psi_{\textrm{pure-bloom}}$ there is a tripartition of output wires $\widetilde{\mathcal{X}} \supseteq \mathcal{X}, \widetilde{\mathcal{Y}} \supseteq \mathcal{Y}$, $\mathcal{Z}$ such that $\mathcal{Z}$ $d$-separates $\widetilde{\mathcal{X}}$ and $\widetilde{\mathcal{Y}}$. Since $\cat{C}$ has conditionals, \Cref{lem:dSepCondIndep} provides us with a decomposition of the form
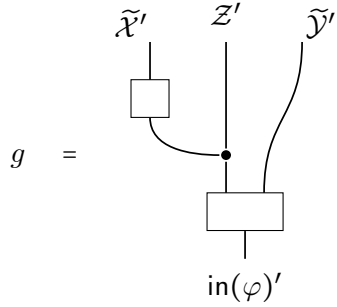
\begin{equation*} 
g \quad = \quad \input{figures/ConditionalIndependenceMorphismTilde.tikz}
\end{equation*} 
By marginalizing over $\widetilde{\mathcal{X}}' \setminus \mathcal{X}'$, $\widetilde{\mathcal{Y}}' \setminus \mathcal{Y}'$ in $g$, we obtain the desired conditional independence for the marginal $f_{\mathcal{X}', \mathcal{Y}', \mathcal{Z}'}$.
\end{proof}

Note that this result includes the soundness of the classical $d$-separation criterion in the classical case of discrete random variables in Bayesian networks\footnote{See \cite{Ve90} for the original proof and \cite[Theorem 1.2.5(i)]{Pe09} for a textbook account.}, which is obtained upon restricting to pure bloom causal models with $\In(\varphi) = \emptyset$, the Markov category $\cat{FinStoch}$, and every box having precisely one output since then conditional independence reduces to \Cref{def:CondIndependence} by \Cref{rem:condInd}.

\begin{definition}\label{def:MarkovProperties}
	Let $\varphi$ be a generalized causal model and $f$ a morphism in a strict Markov category $\cat{C}$ (not necessarily having conditionals). Then we say that $f$ satisfies:
\begin{enumerate}[label=(\roman*)]
	\item the \emph{global Markov property} with respect to $\varphi$ if for every three disjoint sets of outputs $\mathcal{X},\mathcal{Y}$, $\mathcal{Z} \subseteq \Out(\varphi)$ with $\In(\varphi) \subseteq \mathcal{Y} \cup \mathcal{Z}$:
		\begin{equation*}
			\mathcal{X} \text{ and } \mathcal{Y} \text{ are categorically $d$-separated by  } \mathcal{Z} \; \text{ in } \varphi \quad \Longrightarrow \quad \mathcal{X} \perp \mathcal{Y} \ | \ \mathcal{Z} \; \text{ in } f.
		\end{equation*} 
	\item the \emph{local Markov property} with respect to $\varphi$ if for every box $b$ in  $\varphi$, we have
		\begin{equation*}
			\Out(b) \perp \Dec(\Out(b))^{c} \setminus \In(b) \ | \ \In(b) \; \text{ in } f.
		\end{equation*} 
\end{enumerate}
\end{definition}

Note that in the special case of $\varphi$ arising from a DAG, the local and global Markov property agree with the classical definitions of Markov properties \cite[Theorem 1.2.6 and Theorem 1.2.7]{Pe09}.

\begin{theorem}
\label{thm:causalCompdSep}
Suppose that we are given the following:
\begin{itemize}
	\item $\cat{C}$ is a strict Markov category with conditionals.
	\item $\varphi$ is a pure bloom causal model over a hypergraph $\Sigma$ such that the boxes in $\varphi$ have distinct types in $\Sigma$.
	\item $f : \bigotimes_{i=1}^n W'_i \to \bigotimes_{j=1}^m V'_j$ is a morphism in $\cat{C}$.
\end{itemize}
Then the following statements are equivalent:
\begin{enumerate}[label=(\roman*)]
	\item\label{thm_compat} $f$ is compatible with the causal model $\varphi$.
	\item\label{thm_global} $f$ satisfies the global Markov property.
	\item\label{thm_local} $f$ satisfies the local Markov property.
\end{enumerate}
\end{theorem}

\begin{proof}
	$\ref{thm_compat} \Longrightarrow \ref{thm_global}$: The global Markov property is precisely the statement of \Cref{cor:dseparationCondIndependence}.

	$\ref{thm_global} \Longrightarrow \ref{thm_local}$: This follows from the fact that $\Dec(\Out(b))^{c}$ and $\Out(b)$ are $d$-separated by $\In(b)$, and $\In(\varphi) \subseteq \Dec(\Out(b))^{c}$, which makes the global Markov property specialize to the local one.

	$\ref{thm_local} \Longrightarrow \ref{thm_compat}$: We prove this statement by induction over the number of boxes $k \coloneqq |B(G)|$. The case $k = 1$ is trivial. For the step from $k$ to $k+1$, let $b$ be a final box in $\varphi$, which means that $\Dec(\Out(b)) = \Out(b)$. Then, $\varphi$ factorizes as
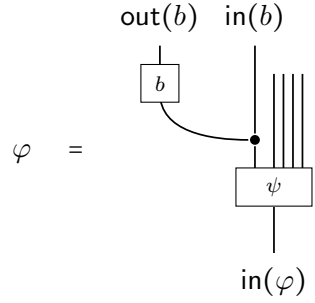
\begin{equation*}
\varphi \quad = \quad \input{figures/proofLocalMarkov0.tikz}
\end{equation*} 
	where $\psi$ is another causal model satisfying all of our assumptions, and no box in $\psi$ has the same type in $\Sigma$ as $b$ does.

	In order to construct a functor $F$ as in \Cref{def:compatibility}, note first that it must satisfy~\eqref{F_wires_map}, which already lets us write the domain of $f$ as $\In(\varphi)'$, and similarly for the codomain.
Since $f$ satisfies the local Markov property with respect to $b$, we can decompose $f$ by \Cref{def:condInd_diff} as
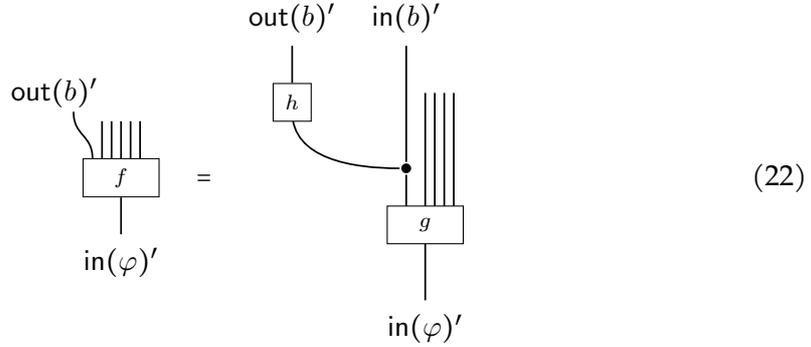
\begin{equation} 
\label{eq:proofLocalMarkov1} \input{figures/proofLocalMarkov1.tikz} \quad = \quad \input{figures/proofLocalMarkov4a.tikz} 
\end{equation}
By induction hypothesis, we have that $g$ is compatible with $\psi$ since $g$ satisfies the local Markov properties specified by $\psi$. Since the box $b$ appears only once in $\varphi$, we can freely define the action of the functor $F$ on $b$ as $F(b) \coloneqq h$. Then, we obtain
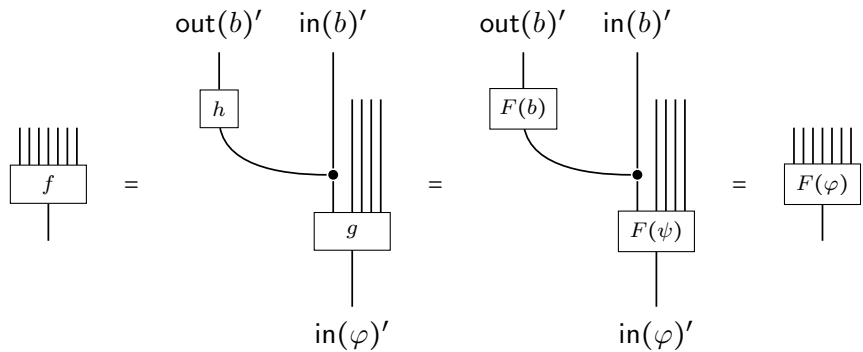
\begin{equation*}
\input{figures/proofLocalMarkov3.tikz} \quad = \quad \input{figures/proofLocalMarkov4a.tikz} \quad = \quad \input{figures/proofLocalMarkov4.tikz} \quad = \quad \input{figures/proofLocalMarkov5.tikz}
\end{equation*} 
where we use in the first step Equation \eqref{eq:proofLocalMarkov1} and in the last the fact that $F$ is a Markov functor.
\end{proof}

\begin{remark}
\begin{enumerate}[label=(\roman*)]
	\item Note that we have used the assumption that $\cat{C}$ has conditionals only for the implication $\ref{thm_compat} \Longrightarrow \ref{thm_global}$. Therefore, for an arbitrary strict Markov category, the global as well as the local Markov property is a sufficient condition for the compatibility of a morphism with a generalized causal model (satisfying our assumptions). However, these Markov properties implicitly require the existence of certain conditionals. Consider, for example, the generalized causal model
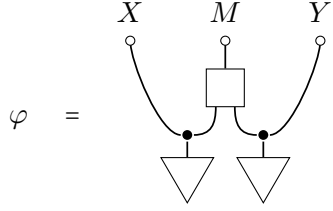
\begin{equation*}
\varphi \quad = \quad \input{figures/emptySeparationExample.tikz}
\end{equation*} 
where all boxes are of distinct types.
Choosing $\mathcal{X} = \{M\}, \mathcal{Y} = \emptyset$ and $\mathcal{Z} = \{X, Y\}$, a morphism $f$ satisfying the global Markov property displays in particular the conditional independence $\{M\} \perp \emptyset \ | \ \mathcal{Z}$, pictorially:
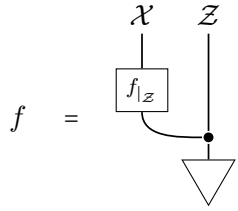
\begin{equation*}
f \quad = \quad \input{figures/emptySeparationExample2.tikz}
\end{equation*} 
This shows that the conditional $f_{|_\mathcal{Z}}$ exists, and this recovers the box that outputs $M$ (up to almost sure equality).
	\item \Cref{thm:causalCompdSep} shows that $d$-separation correctly detects causal compatibility for the Markov categories $\cat{FinStoch}$, $\cat{Gauss}$ or $\cat{BorelStoch}$. For the Markov category $\cat{Stoch}$, which does not have conditionals, the global and local Markov properties are at least sufficient for the compatibility since our proof of these implications has not used conditionals.
	\item Note that \Cref{thm:causalCompdSep} only applies to causal models where each box appears at most once in the model (which in particular implies that $\varphi$ has no nontrivial symmetries). However, the implication $\ref{thm_compat} \Longrightarrow \ref{thm_global}$ applies to arbitrary generalized causal models as proven in \Cref{cor:dseparationCondIndependence}. 
\end{enumerate}
\end{remark}

\begin{example}
\label{ex:dSepExamples}
We now present three examples that go beyond the classical $d$-separation criterion. In (i) we study a causal structure that does not arise from a DAG, in (ii) we study a DAG causal structure with continuous and possibilistic variables, and in (iii) we study a causal structure with inputs with deterministic variables.
\begin{enumerate}[label=(\roman*)]
	\item\label{ex:dSepExamplesI} Let $\varphi$ be the causal structure
	\begin{equation*}
	\input{figures/DAGImpossibleScenarioConcrete.tikz} 
	\end{equation*} 
	and let $\cat{C}$ be a strict Markov category with conditionals. By \Cref{thm:causalCompdSep}, a morphism $t : I \to X' \otimes Z'_1 \otimes Z'_2 \otimes Y'$ in $\cat{C}$ is compatible with this structure if and only if it satisfies
	\begin{equation*}
	X \perp \{Y, Z_2\} \ | \ Z_1 \quad \text{ and } \quad Y \perp \{X, Z_1\} \ | \ Z_2
	\end{equation*} 
	For a general class of examples, consider $X' = Z'_2$ and $Y' = Z'_1$ and any morphism in $\cat{C}$ of the form
	\begin{equation*}
	\input{figures/BipartiteProbability00.tikz}
	\end{equation*} 
	We claim that such a distribution is compatible with $\varphi$ if and only if there exist morphisms $d$ and $d'$ such that
	\begin{equation}
	\label{eq:deterministicCompatibility} \input{figures/BipartiteProbability.tikz} \quad = \quad \input{figures/BipartiteProbabilityDeterministicMorphism.tikz} \quad = \quad \input{figures/BipartiteProbabilityDeterministicMorphism2.tikz}
	\end{equation} 
	where $s$ is the first marginal of $r$, and similarly $d' : Z'_2 \to Z'_1$ satisfies the same equations the other way around.
	Here, the second equation states that the morphism $d$ is $s$-a.s.~deterministic \cite[Definition 13.11]{Fr20}, and similarly for $d'$.
	
	Indeed, assuming compatibility we have that
	\begin{equation} 
	\label{eq:deterministic1}
	\input{figures/BipartiteProbability01.tikz} \quad = \quad \input{figures/BipartiteProbability02.tikz}
	\end{equation} 
	which shows the first equality in Equation \eqref{eq:deterministicCompatibility}. For the second equality we have that
	\begin{equation*}
	\input{figures/BipartiteProbability03.tikz} \quad = \quad \input{figures/BipartiteProbability04.tikz} \quad = \quad \input{figures/BipartiteProbability05.tikz} 
	\end{equation*}
	where we have used the conditional independence $Z_2 \perp X \ | \ Z_1$ in the first step and Equation \eqref{eq:deterministic1} in the second step. Since the morphism is symmetric with respect to permutations of the output wires $X$ and $Z_2$, we have $a = F(f)$ $s$-a.s. which shows the second equality in Equation \eqref{eq:deterministicCompatibility}.
	Proving the existence of $d'$ works analogously by interchanging the roles of $X$ and $Y$ as well as $Z_1$ and $Z_2$.	
	
	Conversely, we have
	\begin{equation*}
	 \input{figures/BipartiteProbability00.tikz} \quad = \quad \input{figures/BipartiteProbability06.tikz} \quad = \quad \input{figures/BipartiteProbability07.tikz} \quad = \quad \input{figures/BipartiteProbability08.tikz}
	 \end{equation*} 	
	where we have used the assumption that $d$ is $s$-a.s.~deterministic in the second equation. Repeating this calculation interchanging the roles of $Z_1$ and $Z_2$ as well as $X$ and $Y$ shows the statement.

	\item Consider the instrumental scenario, given by the DAG

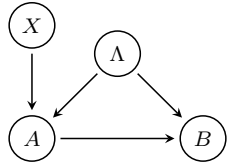
\begin{equation*}
\begin{tikzpicture}[scale=0.75]
	\draw[fill=white, line width=0.2mm] (0,0) circle (0.8cm) node{\scriptsize $A$};
	\draw[fill=white, line width=0.2mm] (6,0) circle (0.8cm) node{\scriptsize $B$};
	\draw[fill=white, line width=0.2mm] (3,3) circle (0.8cm) node{\scriptsize $\Lambda$};
	\draw[fill=white, line width=0.2mm] (0,4) circle (0.8cm) node{\scriptsize $X$};
	
	\draw[-stealth,line width=0.2mm] (0,3) -- (0,1);
	\draw[-stealth,line width=0.2mm] (1,0) -- (5,0);
	\draw[-stealth,line width=0.2mm, rotate around={45:(3,3)}] (3,2) -- (3,-0.2);
	\draw[-stealth,line width=0.2mm, rotate around={-45:(3,3)}] (3,2) -- (3,-0.2);
\end{tikzpicture} 
\end{equation*}
This has been previously studied mainly in the context of DAGs with latent variables \cite{Pe95, Bo01}. For our analysis, we assume each variable to be observed, which means that the causal structure reads string diagrammatically as
\begin{equation} 
\label{eq:PearlStringDiagram}
\varphi \quad = \quad \input{figures/InstrumentalScenario.tikz}
\end{equation} 
There are two non-trivial $d$-separations:
\begin{enumerate}[label=(\alph*)]
	\item Between $X$ and $B$ by $\{A, \Lambda\}$,
	\item Between $X$ and $\Lambda$.
\end{enumerate} 
Therefore, \Cref{thm:causalCompdSep} implies that a distribution $P$ on a four-fold tensor product object in a Markov category with conditionals is compatible with $\varphi$ if and only if $X \perp B \ | \ A, \Lambda$ and $X \perp \Lambda$. 

In $\cat{BorelStoch}$, this means that $P$ is compatible with $\varphi$ if and only if 
\begin{align*}
& P(X \in E_1, A \in E_2, B \in E_3, \Lambda \in E_4) \\ &= \int_{E_2} \int_{E_4} P_{X\, | \, A, \Lambda}(X \in E_1\, | \, a, \lambda) \, P_{B\, | \, A, \Lambda}(B \in E_3\, | \, a, \lambda) \, P_{A,\Lambda}(\mathrm{d} a, \mathrm{d}\lambda)
\end{align*} 
and 
\begin{align*}
P(X \in E_1, \Lambda \in E_4) = P(X \in E_1) \cdot P(\Lambda \in E_4)
\end{align*} 
where $E_i$ are measurable sets in the Borel $\sigma$-algebras of the spaces $X', A', B'$ and $\Lambda'$.

For simplicity, assume that all random variables take values in $\mathbb{R}$ and are absolutely continuous, i.e.~there exists a density $f: X' \times A' \times B' \times \Lambda' \to [0,\infty)$ such that
\begin{equation*}
P(X \in E_1, A \in E_2, B \in E_3, \Lambda \in E_4) = \int_{E_1 \times E_2 \times E_3 \times E_4} f(x,a,b,\lambda) \diff x \diff a \diff b \diff \lambda 
\end{equation*} 

The causal compatibility now amounts to the following two conditions:
\begin{enumerate}[label=(\alph*)]
	\item $X \perp \Lambda$, i.e.
		\begin{equation}
			\label{eq:Ind}
			f_{X,\Lambda}(x,\lambda) = f_{X}(x) \cdot f_{\Lambda}(\lambda) \quad \text{ a.e.}
		\end{equation}
		where a.e.~means almost everywhere with respect to the Lebesgue measure on $\mathbb{R}$.
	\item $X \perp B \ | \ A, \Lambda$, i.e.\
	\begin{equation}
	\label{eq:condInd} 
	f(x,a,b,\lambda) = f_{X\, | \, A,\Lambda}(x,a,\lambda) \cdot f_{A,\Lambda}(a,\lambda) \cdot f_{B\, | \, A,\Lambda}(b,a,\lambda) \quad \text{ a.e.},
	\end{equation}
\end{enumerate}
where the conditional densities are defined implicitly by
\[
f_{A\, | \, X,\Lambda}(a\, | \, x,\lambda) \cdot f_{X,\Lambda}(x,\lambda) = f_{X,A,\Lambda}(x,a,\lambda) \quad \text{ a.e.}
\]
Combining Eq.~\eqref{eq:Ind} and Eq.~\eqref{eq:condInd} results in
\begin{equation*}
f(x,a,b,\lambda) = f_{\Lambda}(\lambda) \cdot f_{X}(x) \cdot f_{A\, | \, X,\Lambda}(a,x,\lambda) \cdot f_{B\, | \, A,\Lambda}(b,a,\lambda) \quad \text{ a.e.}
\end{equation*} 
which is the usual factorization condition for compatibility with the causal structure in \eqref{eq:PearlStringDiagram}.

In $\cat{FinSetMulti}$, the same observation implies that a morphism $$f: I \to X' \otimes A' \otimes B' \otimes \Lambda'$$ is compatible with the causal structure $\varphi$
if and only if the possibility of $X$ and $Y$ can be independently determined, knowing the outcomes of $A$ and $\Lambda$, i.e.\
$$f(x,y \, | \, a, \lambda) = f(x \, | \, a, \lambda) \cdot f(y \, | \, a, \lambda)$$
and the possibilities of $X$ and $\Lambda$ are independent, i.e.\ $f(x,\lambda) = f(x) \cdot f(\lambda)$.

\item Consider the following string diagram

$$ \varphi \quad = \quad \input{figures/deterministicStringDiagram.tikz} $$

A morphism $f$ in $\cat{FinSet}$ is compatible with $\varphi$ if the variables $z \in Z, x \in X$ and $y \in Y$ and  arise via functions $g: C \to Z$, $h: A \times Z \to X$, and $k: B \times Z \to Y$, i.e.\
$$ z = g(c), \quad x = h(a, z) \quad \text{and} \quad y = k(b,z).$$

By \Cref{thm:causalCompdSep}, $f$ is compatible with $\varphi$ if the following conditions are satisfied (due to the local Markov property):
\begin{enumerate}[label=(\alph*)]
	\item\label{CondItemA} $Z \perp \{A, B\} \ | \ C$
	\item\label{CondItemB} $X \perp \{B,C\} \ | \ \{A, Z\}$
	\item\label{CondItemC} $Y \perp \{A, C\} \ | \ \{B, Z\}$
\end{enumerate}
By \Cref{ex:CondIndSet}, Condition \ref{CondItemA} implies the existence of a function $g$ such that $z = g(c)$, Condition \ref{CondItemB} implies $x = h(a,z)$, and Condition \ref{CondItemC} implies $y = k(a,z)$. But this directly shows the compatibility of $f$ with the causal model $\varphi$. The local Markov property immediately implies causal compatibility in $\cat{FinSet}$ since the construction of conditionals (see \Cref{prop:SetMultiConditionals}) is trivial.
\end{enumerate}
\end{example}

\begin{question}
	Can \Cref{thm:causalCompdSep} be extended to more general causal models? In particular, what about allowing the same box to appear several times in $\varphi$?
\end{question}

\bibliographystyle{abbrv}
\bibliography{Bibliography}

\end{document}